\documentclass{amsart}
\usepackage{amsmath}
\usepackage{color}
\usepackage{graphicx}
\usepackage{framed}

\newtheorem{thm}{Theorem}
\newtheorem{assumption}[thm]{Assumption}
\newtheorem{setting}[thm]{Setting}
\newtheorem{lem}[thm]{Lemma}

\newtheorem{prop}[thm]{Proposition}
\newtheorem{rk}[thm]{Remark}

\newcommand{\vip}{\vskip.2cm}
\newcommand{\e}{{\varepsilon}}
\newcommand{\cX}{{\mathcal X}}
\newcommand{\cH}{{\mathcal H}}
\newcommand{\cT}{{\mathcal T}}
\newcommand{\cA}{{\mathcal A}}
\newcommand{\cZ}{{\mathcal Z}}
\newcommand{\cY}{{\mathcal Y}}
\newcommand{\cB}{{\mathcal B}}
\newcommand{\cF}{{\mathcal F}}
\newcommand{\cI}{{\mathcal I}}
\newcommand{\cG}{{\mathcal G}}
\newcommand{\cM}{{\mathcal M}}
\newcommand{\rr}{{\mathbb{R}}}
\newcommand{\Q}{{\mathbb{Q}}}
\newcommand{\HH}{{\mathbb{H}}}
\newcommand{\bF}{{{\bf F}}}
\newcommand{\bQ}{{\bar Q}}
\newcommand{\bpi}{{\bar \pi}}
\newcommand{\tV}{{\tilde V}}
\newcommand{\tQ}{{\tilde Q}}
\newcommand{\tN}{{\tilde N}}
\newcommand{\ttau}{{\tilde \tau}}
\newcommand{\te}{{\tilde \e}}
\newcommand{\tG}{{\tilde G}}
\newcommand{\tpi}{{\tilde \pi}}
\newcommand{\hQ}{{\hat Q}}
\newcommand{\hf}{{\hat f}}
\newcommand{\hphi}{{\hat \varphi}}
\newcommand{\hpi}{{\hat \pi}}
\newcommand{\argmax}{{\rm arg\,max}\;}
\newcommand{\dd}{{\rm d}}
\newcommand{\nn}{{\mathbb{N}}}
\newcommand{\ig}{[\![}
\newcommand{\id}{]\!]}
\newcommand{\zz}{e}
\newcommand{\E}{\mathbb{E}}
\newcommand{\PP}{\mathbb{P}}
\newcommand{\indiq}{{{\mathbf 1}}}

\newcommand{\black}{\color{black}}
\newcommand{\blue}{\color{black}}

\begin{document}

\title[Markov decision processes]{Markov decision processes: on the convergence of the Monte-Carlo 
first visit algorithm}

\author{Sylvain Delattre}

\author{Nicolas Fournier}

\address{Sylvain Delattre, Université Paris Cité and Sorbonne Université, CNRS, Laboratoire de 
Probabilités, Statistique et Modélisation, F-75013 Paris, France.}

\email{sylvain.delattre@lpsm.paris}

\address{Nicolas Fournier,  Sorbonne Université and Université Paris Cité, CNRS, 
Laboratoire de Probabilités, Statistique et Modélisation, F-75005 Paris, France.}

\email{nicolas.fournier@sorbonne-universite.fr}

\begin{abstract}
We consider the Monte-Carlo first visit algorithm, of which the goal is to find the optimal control in
a Markov decision process with finite state space and finite number of possible actions.
We show its convergence when the discount factor is smaller than $1/2$.
\end{abstract}



\thanks{We warmly thank the referees for their fruitful comments}

\maketitle

\section{Introduction}

\blue
This paper deals with
Markov decision processes (MDP) with finite states/actions set.
A MDP is a model where an agent interacts with its environment: at each step, the agent lies in some state $x$, 
chooses some action $a$, and the environment chooses at random, with some law depending on $(x,a)$, 
the instantaneous reward and the next state to which the agent goes. 
The goal is to find a policy (i.e. a way to choose the actions given the states)
maximizing the expectation of $\sum_{t\geq 1} \gamma^{t-1} R_{t}$, where
$R_{t}$ is the instantaneous reward at step $t\geq 1$ and where $\gamma\in [0,1]$ is the discount factor.

\vip

There are mainly three types of model-free algorithms used to find an optimal policy: 
the so-called $Q$-learning algorithm,
the so-called TD (temporal difference) algorithm and its variant TD$(\lambda)$, and the so-called 
Monte-Carlo algorithms (first visit and every visit). All these algorithms are model free, 
in that we do not suppose that we know the transition probabilities: we
only assume we can interact with the MDP by producing some episodes, given a policy. 
We refer to Sutton-Barto \cite{sb} for a complete and 
accurate exposition on all these algorithms. In particular, the convergence of the $Q$-learning algorithm
(see Watkins-Dayan~\cite{wd}) and of the TD-algorithm and its variants (see Sutton~\cite{s}, Dayan~\cite{d}) 
has been established for any value of 
$\gamma\in [0,1)$, provided there is enough exploration. All these results rely on the fact that there is some
underlying approximate $\gamma$-contraction in the space of the value functions (from the set of couples 
state/action into $\rr$). On the contrary, the convergence of the Monte-Carlo algorithms is still widely open.
See however the partial results of Tsitsiklis~\cite{t} and Wang-Yuan-Shao-Ross~\cite{wysr} that 
we will discuss below.
Let us quote Sutton-Barto \cite[page 99]{sb},  discussing about the convergence of the FVA/EVA:
``In our opinion, this is one of the most fundamental open
theoretical questions in reinforcement learning (for a partial solution, see Tsitsiklis, 2002).''

\vip

The main difference between the Monte-Carlo algorithms and the other ones is that they
do not bootstrap, in the sense that they do not use the approximate value function to update 
this value function. They use instead a non-biased estimation of the total reward, which explains their robustness.
In particular, they are still used in the modern context of the function approximation (meaning that the
states/actions space is very large so that we have to approximate the value functions using a parameterized
family of functions), while it is well-known that the other algorithms may fail to converge in such a context,
see Baird's counterexample~\cite{baird} and the discussion of Silver~\cite{silv}.

\vip

Monte-Carlo algorithms are based on simulating a sequence
of episodes. We modify, after 
each episode, the policy used to produce the next episode.
There are two main available Monte-Carlo algorithms: the {\it first visit algorithm} (FVA) and 
the {\it every visit algorithm} (EVA), see \cite[Section 5.1]{sb}. In the FVA (resp. EVA), 
we update the policy by estimating the value function at each pair $(x,a)$ of state/action using only the first 
visit at $(x,a)$ (resp. all the visits at $(x,a)$) of each episode.

\black
\vip

We study here the FVA, which is easier than the EVA from a theoretical point of view.
We are able to prove that when $\gamma \in [0,1/2)$, 
the FVA produces some policy of which the value function converges to the optimal one.
Actually, we consider a more general family of (abstract) algorithms and show with a counter
example (strongly inspired by Example 5.12 in Bertsekas-Tsitsiklis \cite{bt}) that the convergence 
cannot hold true when \blue $\gamma > 1/2$ \black for this general family.
However, the counter example is far from concerning the FVA: it only shows that our 
proof breaks down, not at all that the FVA does not converge.

\vip

Tsitsiklis \cite{t} studies, among other algorithms, 
a {\it synchronous optimistic policy iteration algorithm}, which is a simplified synchronous
version of the FVA: each iteration consists of simulating many trajectories, one per couple
state/action. More or less, this implies that the \blue learning rate  $\alpha_k(x,a)$ is deterministic 
and does not depend on $(x,a)$ in Algorithm 1 below.
He shows the convergence of his algorithm for any 
$\gamma \in [0,1)$. The proof we will handle in the present paper is of the same spirit as that of 
\cite{t}, with a number of complications, mainly because our learning rates depend on the states/actions. 
In particular,
an important step of our proof consists in showing an approximate contraction property 
with parameter $\gamma/(1-\gamma)$, which is smaller than $1$ only when $\gamma \in [0,1/2)$, 
see Lemma \ref{contr}.
\black

\vip
Liu \cite{l} extends the result of \cite{t} to the case $\gamma=1$, assuming that all the policies are proper:
some final state space is reached with probability $1$, starting from any state and using any policy. 
\blue He replaces the contraction property used in \cite{t} by some monotony properties.\black

\vip

Wang-Yuan-Shao-Ross \cite{wysr} show the convergence of the FVA, with $\gamma=1$ (but, as they mention, 
this extends to any $\gamma \in [0,1]$), 
assuming a structural condition on the MDP: it has to be {\it optimal policy 
feed forward}, meaning that a state cannot be visited twice under any optimal policy.
\blue Their proof strongly relies on the fact that in such a case, the MDP almost has a bounded horizon, 
so that they can work by backward induction.\black

\vip

\blue 
It thus seems that the present paper provides the first convergence proof of the FVA without assuming 
any structural condition on the underlying MDP. However, only half the way is done, 
since we can only treat the case $\gamma\in [0,1/2)$, and there is no hope that
our proof could extend to the case $\gamma > 1/2$. Our proof does not provide any rate of convergence, 
but we are not aware of any result of this kind for Monte-Carlo control algorithms.
\vip

\blue 
Of course, many real-life MDPs have an infinite state-space. Even in the countable case, we cannot hope 
to explorate all the states. 
Hence the use of some $\max$ norm (as in the present work) to compare the value function and its 
approximation is irrelevant. Thus one has to use another norm, of which the shape necessarily strongly 
depends on the MDP under study. A general study seems extremely delicate.

\black

\blue 
\subsection*{Plan of the paper}
In the next section, we describe the model under consideration, recall the main known results about optimal
policies, introduce the algorithms to be studied and present our main results.
In Section \ref{cont}, we prove a crucial contraction result, with constant $\gamma/(1-\gamma)$, which 
is smaller than $1$ if and only if $\gamma <1/2$.
Section \ref{agp} is devoted to the proof of an abstract convergence result.
This abstract result is applied to show the convergence of a general algorithm
in Section \ref{convgen}. In Section \ref{convfv}, we show how to apply this convergence result to the 
first-visit algorithm. We quickly study the case with finite episodes in Section \ref{mrs}.
In Section \ref{counterex}, we show, through a counter-example, that our proof 
cannot be extended to the case where $\gamma \in (1/2,1)$.
Finally, we quickly recall in Appendix \ref{qpkr}, for the sake of completeness, 
the proofs of some well-known results about 
Markov decision processes and of a simple version of the Robbins-Monro lemma.

\black

\section{Notation and results}\label{nmr}

\blue For $k,\ell \in \nn$ with $k\leq \ell$, we use the notation $\ig k,\ell\id=\{k,k+1,\dots,\ell\}$. \black

\subsection{The model}\label{mod}

The following objects are fixed in the whole paper.

\begin{setting}\label{set}
Let $\cX$ be a non-empty finite state space and, for each $x\in \cX$, let $\cA_x$ be a non-empty 
finite set of possible
actions. We also set $\cZ=\{(x,a) : x \in \cX, a \in \cA_x\}$. 
For each $(x,a) \in \cZ$,
we consider a probability measure $P(x,a,\cdot)$ on $\cX$. We also consider,
for each $(x,a) \in \cZ$ and $y \in \cX$, a probability measure $S(x,a,y,\cdot)$ on $\rr$ satisfying
$\int_\rr z^2 S(x,a,y,\dd z)<\infty$ and we set
\begin{equation}\label{ttt}
g(x,a,y)=\int_\rr zS(x,a,y,\dd z).
\end{equation}
Finally, we consider a real number $\gamma \in [0,1)$.
\end{setting}
\vip

For $t \in\nn$, we denote by $\cB_t=\{(x_0,a_0,\dots,x_{t-1},a_{t-1},x_t) : 
x_i \in \cX, a_i \in \cA_{x_i}\}$.
A {\it policy} is a family $\Pi=(\Pi_t)_{t\geq 0}$, where for each $t\geq 0$, for each 
$h=(x_0,a_0,\dots,x_{t-1},a_{t-1},x_t)\in \cB_t$, $\Pi_t(h,\cdot)$ is a probability measure on 
$\cA_{x_t}$.

\vip

A policy $\Pi=(\Pi_t)_{t\geq 0}$ is said to be SM (for stationary Markov) if there is a family
$(\pi(x,\cdot))_{x\in \cX}$, with $\pi(x,\cdot)$ a probability measure on $\cA_x$, such that for all 
$t\geq 0$, all $h=(x_0,a_0,\dots,x_{t-1},a_{t-1},x_t)\in \cB_t$, $\Pi_t(h,\cdot)=\pi(x_t,\cdot)$.
In such a case, we simply say that $\pi=(\pi(x,\cdot))_{x\in \cX}$ is a SM policy.

\vip

Given a starting point $x \in \cX$ and a policy $\Pi$, we build recursively the stochastic process 
$(X_t,A_t)_{t\in \nn}$ as follows:
we set $X_0=x$, we build $A_0$ with law $\Pi_0(X_0,\cdot)$ and, assuming that we have built 
$(X_s,A_s)_{s \in \ig 0,t-1\id}$ for some $t\geq 1$, we first build
$X_t$ with conditional law $P(X_{t-1},A_{t-1},\cdot)$ knowing  $(X_s,A_s)_{s \in \ig 0,t-1\id}$, we set
$H_t=(X_0,A_0,\dots,X_{t-1},A_{t-1},X_t)$ and we build $A_t$ 
with conditional law $\Pi_t(H_t,\cdot)$ knowing 
$H_t$. 

\vip

At each step, there is a reward: conditionally on the whole process 
$(X_s,A_s)_{s\in \nn}$ the family of
rewards $(R_s)_{s \in \nn_*}$ is independent and for each $t\geq 1$, the reward $R_t$ is
$S(X_{t-1},A_{t-1},X_t,\cdot)$-distributed. The total reward is then given by
\begin{equation}\label{G}
G=\sum_{t \in \nn} \gamma^t R_{t+1}\qquad \hbox{(convention: $0^0=1$).}
\end{equation}

We indicate in subscript of the probability $\PP_{x,\Pi}$ and expectation $\E_{x,\Pi}$
the starting point $x\in \cX$ and the policy $\Pi$ used to build the above random variable $G$,
and we consider the value function
\begin{equation}\label{vpi}
V_\Pi(x)=\E_{x,\Pi}[G].
\end{equation}
For $\pi$ a SM policy, we simply denote by $\PP_{x,\pi}$, $\E_{x,\pi}$ and $V_\pi(x)$ the corresponding objects. 
Finally,
we set, for $x\in \cX$,
$$
V^*(x)= \sup \{V_\Pi(x) , \; \Pi \;\;\hbox{policy}\}.
$$

\subsection{Optimal policy}\label{opol}
The existence of an optimal policy, which is SM and does not depend on the 
starting point, is well known, see Puterman \cite{p} and Sutton-Barto \cite{sb}.
\blue 
Things can be summarized as follows. For $(x,a)\in \cZ$, we set 
\begin{equation}\label{qstar1}
r(x,a)=\sum_{y\in \cX} P(x,a,y) g(x,a,y),
\end{equation}
which stands for the mean (instantaneous) reward, when the process lies in state $x$ and when one chooses
the action $a$. For $\pi$ a SM policy, for $x,y\in \cX$ and $a\in\cA_x$, we set 
\begin{gather}
r_\pi(x)=\sum_{a \in \cA_x}r(x,a)\pi(x,a), \qquad
P_\pi(x,y)=\sum_{a \in \cA_x}P(x,a,y)\pi(x,a), \label{rpiPpi} \\
\hbox{and} \quad Q_\pi(x,a)=r(x,a)+\gamma PV_\pi(x,a), \label{qpi}
\end{gather}
where $PV_\pi(x,a)=\sum_{y\in \cX} P(x,a,y)V_\pi(y)$. Observe that 
$r_\pi(x)$ represents the mean (instantaneous) reward, 
when the process lies in state $x$ and when using the policy $\pi$, that $P_\pi$ is the transition matrix
of the process $(X_t)_{t\in\nn}$ when using the policy $\pi$,
while $Q_\pi(x,a)$ stands for the mean total reward, when starting from the state $x$, when choosing
$a$ as first action, and when using the policy $\pi$ for the rest of the process. Finally, we introduce
\begin{equation}\label{qstar}
Q^*(x,a)=r(x,a)+\gamma PV^*(x,a), \qquad \hbox{where}\qquad PV^*(x,a)=\sum_{y\in \cX} P(x,a,y)V^*(y).
\end{equation}
\black

\begin{thm}\label{known}
(i) Consider a SM policy $\pi^*$ such that
\begin{equation}\label{trruc}
\pi^*\Big(x, \argmax Q^*(x,\cdot)\Big)=1 \quad \hbox{for all } x\in \cX.
\end{equation}
Then $V_{\pi^*}(x)=V^*(x)$ for all $x\in\cX$.
\vip
(ii) For $\pi$ a SM policy, we have $Q_\pi(x,a)\leq Q^*(x,a)$ for all $(x,a)\in \cZ$ and, for all $(x,a)\in \cZ$,
\begin{gather*}
V_\pi(x)=r_\pi(x)+\gamma P_\pi V_\pi (x)= \sum_{a\in \cA_x} Q_\pi(x,a)\pi(x,a),\\
Q_\pi(x,a)=r(x,a)+\gamma \sum_{(y,b)\in \cZ} P(x,a,y)\pi(y,b)Q_\pi(y,b)
\blue =r(x,a)+\gamma \sum_{y\in\cX} P(x,a,y) V_\pi(y). \black
\end{gather*}

(iii) If $\pi^*$ satisfies \eqref{trruc}, then $Q_{\pi^*}=Q^*$. Moreover, for all $(x,a)\in \cZ$
\begin{gather*}
V^*(x)=\max_{a \in \cA_x} Q^*(x,a), \\
Q^*(x,a)=r(x,a)+\gamma \sum_{y\in \cX} P(x,a,y) \max_{b \in \cA_y} Q^*(y,b)
\blue = r(x,a)+\gamma\sum_{y\in \cX} P(x,a,y) V^*(y).\black
\end{gather*}
\end{thm}

\blue The proof of Theorem~\ref{known} is recalled in Appendix \ref{qpkr} for the sake of completeness.
This shows that there exists an optimal policy $\pi^*$ which is SM. Naturally, the corresponding $Q_{\pi^*}$ 
is nothing but $Q^*$. 
Formula \eqref{trruc} explains that the SM policies which choose (at $x\in\cX$) some action $a\in\cA_x$ maximizing
$Q^*(x,\cdot)$ are optimal.
Since there exists an optimal policy which is SM, we may and will, from now on, consider only some SM policies.
\black
 
\subsection{A general class of algorithms}\label{gca}
We consider the following family of algorithms.
\begin{oframed}
\centerline{ \blue Algorithm 1 (general algorithm) \black}
\vip\vip
\noindent Start with some deterministic function $\hQ^0 : \cZ \to \rr$. 
For $k\geq 0$, assume that $k$ first episodes
$$
(X^i_0,A^i_0,X^i_1,R^i_1,A^i_1,X^i_2,R^i_2,A^i_2,\dots), \quad i=1,\dots,k,
$$
have already been built,
and consider the sigma-field $\cF^{k}$ generated by all these random variables (with $\cF^0=\{\emptyset,\Omega\}$).
We also introduce
$$
\cG^{i}_n=\sigma(X^{i}_0,A^{i}_0,X^{i}_1,R^{i}_1,A^{i}_1,\dots,X^{i}_n,R^{i}_n,A^{i}_n),
$$
so that $\cF^k=\lor_{i=1}^{k} \cG^i_\infty$.
Assume also that the $\cF^{k}$-measurable (random) function $\hQ^{k} : \cZ \to \rr$ has been built.
For some $\cF^{k}$-measurable (random) function $\e_{k}:\cX \to [0,1]$, we define the policy
\begin{equation}\label{gg1}
\hpi^{k}(x,a)= \frac{\e_{k}(x)}{|\cA_x|} + \indiq_{\{a\in \argmax \hQ^{k}(x,\cdot)\}} 
\frac{1-\e_{k}(x)}{|\argmax \hQ^{k}(x,\cdot)|}.
\end{equation}
For some $\cF^{k}$-measurable (random) probability $\nu_{k}$ on $\cZ$, we
build the $(k+1)$-th episode 
$$
(X^{k+1}_0,A^{k+1}_0,X^{k+1}_1,R^{k+1}_1,A^{k+1}_1,X^{k+1}_2,R^{k+1}_2,A^{k+1}_2,\dots)
$$
using $(X_0^{k+1},A^{k+1}_0)\sim \nu_{k}$ and then the SM policy $\hpi^k$, as in Subsection \ref{mod}.
For each $(x,a) \in \cZ$, we consider the $(\cG^{k+1}_n)_{n\geq 0}$-stopping-time 
$$
\tau^{k+1}_{x,a}=\inf \{ t \geq 0 : (X^{k+1}_t,A^{k+1}_t)=(x,a)\} \qquad \hbox{(convention: $\inf \emptyset=\infty$).}
$$
For all $(x,a) \in \cZ$, for 
some $\cF^{k} \lor \cG^{k+1}_{\tau^{k+1}_{x,a}}$-measurable 
random variable $\alpha_{k}(x,a) \in [0,1]$, we set
\begin{gather}\label{gg3}
\hQ^{k+1}(x,a)=(1-\alpha_{k}(x,a)\indiq_{\{\tau^{k+1}_{x,a}<\infty\}})\hQ^{k}(x,a) + 
\alpha_{k}(x,a)\indiq_{\{\tau^{k+1}_{x,a}<\infty\}}G^{k+1}_{x,a},\\
\label{gg2}\hbox{where}\qquad G^{k+1}_{x,a}= \sum_{t \in \nn} \gamma^t R^{k+1}_{\tau^{k+1}_{x,a}+t+1}
\qquad \hbox{\blue (convention : $G^{k+1}_{x,a}=0$ when $\tau^{k+1}_{x,a}=\infty$).\black}
\end{gather}
\end{oframed}

\blue
This algorithm depends on the choices of the initial function $\hQ^0$,
on the families of random functions $(\e_k)_{k\geq 0}$ (the exploration/exploitation trade-offs)
and $(\alpha_k)_{k\geq 0}$ (the learning rates), and of the family of 
probability measures $(\nu_k)_{k\geq 0}$ (the initial laws of the state/action).
Let us explain why $\hQ^k$ defined in Algorithm 1 might converge to  $Q^*$ 
(which implies that $V_{\hpi^k}$ converges to $V^*$).
For $q:\cZ \to \rr$, let us introduce the $q$-greedy policy
$$
\pi_q(x,a)=\indiq_{\{a \in \argmax q(x,\cdot)\}} \frac{1}{|\argmax q(x,\cdot)|}\qquad \hbox{for all $(x,a)\in\cZ$.}
$$
As is well-known, $Q^*$ is characterized by the identity $Q^* = Q_{\pi_{Q^*}}$, see Theorem \ref{known}.
Since we have $\E[G^{k+1}_{x,a}\mid \cF_k] = Q_{\hat\pi^k}(x,a)$ when $(x,a)$ is visited, 
we hope that \eqref{gg3} makes
$\hQ^{k+1}$ close to $Q_{\hat\pi^k}$. Since moreover $\hat\pi^k \simeq \pi_{\hQ^k}$ 
(with some additional small $\epsilon$ for exploration), we hope that if $\hQ=\lim_k \hQ^k$ exists, 
then $\hQ=Q_{\pi_{\hQ}}$, whence $\hQ=Q^*$.
\black

\vip

As we will see in the next subsection, for some specific choice of \blue the parameters, Algorithm 1 \black is 
the so-called first-visit algorithm.

\begin{thm}\label{mrg}
Assume that $\gamma \in [0,1/2)$. Consider \blue Algorithm 1 \black
and assume that
\begin{align}\label{cqdev}
\left\{\begin{array}{c}
\hbox{a.s., for all $(x,a)\in \cZ$,} \quad
\lim_{k\to \infty} \e_k(x)=0,\quad \sum_{k\geq 1} \alpha_k(x,a)\indiq_{\{\tau^{k+1}_{x,a}<\infty\}}=\infty, \\[8pt]
\hbox{and}\quad \sum_{k\geq 1} (\alpha_k(x,a))^2\indiq_{\{\tau^{k+1}_{x,a}<\infty\}}<\infty.
\end{array}\right.
\end{align}
Then a.s., $\lim_{k\to \infty} V_{\hpi^k}(x)=V^*(x)$ for all $x\in\cX$.
\end{thm}

\blue 
In \cite{t}, Tsitsiklis has shown a similar result for every $\gamma\in[0,1)$, when $\alpha_k(x,a)$ 
does not depend on $(x,a)$ and
under the assumption that all the states/actions are visited at each episode. 
His class of algorithm does not contain the FVA.
\black

\vip

Although this is not very convincing as far as the first-visit algorithm is concerned,
the condition that $\gamma \in [0,1/2)$, or at least that $\gamma \in [0,1/2]$, 
is necessary in the above theorem. This is strongly inspired by Example 5.12 in cite \cite{bt}.

\begin{prop}\label{ptb}
Consider the simplest possible Markov decision process: take a state space $\cX=\{\zz\}$ with one element,
take $\cA_0=\{0,1\}$, so that necessarily $P(\zz,0,\zz)=P(\zz,1,\zz)=1$,
and take $S(\zz,0,\zz,\cdot)=\delta_0$ and $S(\zz,1,\zz,\cdot)=\delta_1$. 
If $\gamma\in (1/2,1)$, we can design the parameters of \blue Algorithm 1 \black
(with in particular $\e_k(\zz)=0$ for all $k\geq 0$)
such that \eqref{cqdev} holds true
but $\lim_{k\to \infty} V_{\hpi^k}(\zz)$ does a.s. not exist.
\end{prop}

\blue 
Let us emphasize that although $\epsilon_k=0$, \eqref{cqdev} is satisfied: all the 
states/actions are visited an infinite
number of times. The non-convergence is not due to a lack of exploration.
\black

\vip

\blue We do not know what happens when $\gamma=1/2$, but it is possible that our proof might be adapted with a 
lot of work, in the spirit of Liu \cite{l}, who extended the result of Tsitsiklis \cite{t} 
(concerning another algorithm and relying on a $\gamma$-contraction, when $\gamma \in [0,1)$) 
to the case $\gamma=1$,
using some monotony arguments. However, the case $\gamma=1/2$ is not particularly interesting and does not warrant
a special treatment. \black

\subsection{The first-visit algorithm}
The algorithm presented in this section is very classical in reinforcement learning,
see Singh-Sutton \cite{ss} and Sutton-Barto \cite{sb}.

\begin{oframed}
\centerline{ \blue Algorithm 2 (first visit algorithm) \black}
\vip\vip
\noindent We consider $\theta>0$ and a law $\mu_0$ on $\cX$. We set $\bQ^0(x,a)=0$ for all $(x,a)\in\cZ$.
For $k\geq 0$, assume that $k$ first episodes
$$
(X^i_0,A^i_0,X^i_1,R^i_1,A^i_1,X^i_2,R^i_2,A^i_2,\dots), \quad i=1,\dots,k,
$$
have been built, as well as the function $\bQ^{k}:\cZ\to\rr$. For $(x,a)\in\cZ$, set
\begin{gather}\label{fv0}
N_k(x,a)=\sum_{i=1}^k \indiq_{\{\tau^i_{x,a}<\infty\}} \qquad \hbox{and} \quad N_k(x)=\sum_{a \in \cA_x} N_k(x,a)
\quad \hbox{(convention: $\sum_{i=1}^0=0$),}\\
\hbox{where} \qquad \tau^i_{x,a}=\inf\{t\geq 0 : (X^i_t,A^i_t)=(x,a)\}\qquad \hbox{(convention: 
$\inf \emptyset=\infty$).}\notag
\end{gather}
Consider the SM policy defined, for $(x,a)\in\cZ$, by
\begin{equation}\label{fv1}
\bpi^{k} (x,a)= \frac{\e_k(x)}{|\cA_x|} + \indiq_{\{a \in \argmax \bQ^k(x,\cdot)\}} 
\frac{1-\e_k(x)}
{|\argmax \bQ^k(x,\cdot)|},\; \hbox{where}\; \e_k(x)=\frac1{(1+N_k(x))^{\theta}}.
\end{equation}

We then build the $(k+1)$-th episode 
$$
(X^{k+1}_0,A^{k+1}_0,X^{k+1}_1,R^{k+1}_1,A^{k+1}_1,X^{k+1}_2,R^{k+1}_2,A^{k+1}_2,\dots )
$$
using $X^{k+1}_0\sim \mu_0$ and the policy $\bpi^{k}$, as in Subsection \ref{mod}.
We then compute, for each $(x,a)\in \cZ$,
\begin{gather}\label{fv2}
\bQ^{k+1}(x,a)=\frac{\sum_{i=1}^{k+1} \indiq_{\{\tau^i_{x,a}<\infty\}}G^i_{x,a}}
{N_{k+1}(x,a)} \qquad \hbox{(convention: $\frac00=0$)},\\
\hbox{where} \qquad G^{i}_{x,a}= \sum_{t \in \nn} \gamma^t R^{i}_{\tau^{i}_{x,a}+t+1}
\qquad \hbox{(convention : $G^{i}_{x,a}=0$ when $\tau^i_{x,a}=\infty$).}
\end{gather}
\end{oframed}

\blue As we will see in Section \ref{convfv}, Algorithm 2 is
a particular case of Algorithm 1, namely when choosing $\hQ^0=0$ and, for all $k\geq 0$ and $(x,a)\in\cZ$, 
$\nu_k(x,a)=\mu_0(x)\hpi^k(x,a)$, $\e_k(x)=(1+N_k(x))^{-\theta}$ and $\alpha_k(x,a)=(N_{k+1}(x,a))^{-1}$.
When $N_{k+1}(x,a)=0$, $\alpha_k(x,a)$ is not defined but this is not an issue as Algorithm 1 uses
$\alpha_k(x,a)$ only when $N_{k+1}(x,a)>0$. \black

\vip

\blue

Here is a pseudo-code for Algorithm 2:
here $Q(x,a)$ (resp $n(x,a)$, $N(x)$, $\pi(x,a)$, $\e(x)$) represents $\bar Q^k(x,a)$ 
(resp. $N_k(x,a)$, $N_k(x)$, $\bar \pi^k(x,a)$, $\e_k(x)$).

\blue
\begin{oframed}
\centerline{Pseudo-code for Algorithm 2}
\vip\vip
{\tt
set $k=0$
\vip
for all $(x,a) \in\cZ$, set $Q(x,a)=0$, $n(x,a)=0$, $N(x)=0$
\vip
do \hskip0.35cm for all $x \in \cX$, set $\e(x)=(1+N(x))^{-\theta}$
\vip
...  for all $(x,a) \in\cZ$, set 
$\pi(x,a)=\frac{\e(x)}{|\cA_x|} + \indiq_{\{a \in \argmax Q(x,\cdot)\}} \frac{1-\e(x)}{|\argmax Q(x,\cdot)|}$ 
\vip
... simulate $X_0,A_0,X_1,R_1,A_1,X_2,R_2,A_2,\dots$ with $X_0\sim\mu_0$ using the policy $\pi$
\vip
... for all $(x,a)\in \cZ$, if there is $t\in \nn$ such that $(X_t,A_t)=(x,a)$
\vip
... ... find $\tau=\min\{t \in \nn : (X_t,A_t)=(x,a)\}$
\vip
... ... set $G=\sum_{s\in \nn}\gamma^s R_{\tau+s}$ and $Q(x,a)=(n(x,a)Q(x,a)+G)/(n(x,a)+1)$
\vip
... ... set $n(x,a)=n(x,a)+1$ and $N(x)=N(x)+1$
\vip
... set $k=k+1$
\vip
until $k$ has reached the desired value
}
\end{oframed}

\black

\begin{thm}\label{mmrr}
\blue Consider the notation introduced in Algorithm 2. \black Assume that $\gamma \in [0,1/2)$. 
If $\theta\in (0,1]$ and $\mu_0(x)>0$ for all $x\in \cX$, then a.s., 
$\lim_k V_{\bpi^k}(x)=V^*(x)$ for all $x\in \cX$.
\end{thm}

\blue
We believe that the condition $\gamma<1/2$ is not necessary for this result to hold true. If e.g. one assumes 
that the MDP is {\it optimal policy feed forward} (meaning that a state cannot be visited twice under 
any optimal policy), the convergence has been shown by Wang-Yuan-Shao-Ross \cite{wysr} 
for any $\gamma \in (0,1]$. However, since
our proof relies on Theorem~\ref{mrg}, which is false when $\gamma>1/2$ by Proposition~\ref{ptb}, there is no
hope that our method extends to the case $\gamma>1/2$, and we have no idea to treat this case. However, it seems
that Theorem~\ref{mmrr} is the first result showing the convergence of the FVA without any structural condition on 
the underlying MDP.

\subsection{A more realistic situation} \label{smrs}

In the previous subsection, we were using some episodes with infinite length. This is of course not realistic.
We now consider the following common situation. For $x,y\in\cX$, we set 
\blue $\bar P(x,y)=\max_{a\in\cA_x}P(x,a,y)$.\black

\begin{assumption}\label{tri}
There is a subset $\triangle$ of $\cX$ such that $S(x,a,y,\dd z)=\delta_0(\dd z)$ 
for all $x\in \Delta$, all $a \in\cA_x$, all $y \in \cX$
and such that $P(x,a,\triangle)=1$ for all $x\in \triangle$, all $a\in \cA_x$.
Moreover, for all $x \in\cX$, there are $y \in \triangle$ and $n\geq 0$ such that
\blue $\bar P^n(x,y)>0$. \black
\end{assumption}

\blue Here $\bar P^n$ is the $n$-th (matrix) power of $\bar P$. \black
In other words, the reward is identically equal to $0$ when the process lies in $\triangle$, whatever the action. 
Moreover, when the process reaches $\triangle$, it remains stuck in $\triangle$ forever.
Finally, it is possible to reach $\triangle$ from anywhere, at least
if choosing suitably the actions.

\vip

\blue 
Consider any game with one player, with a finite number of states and a finite number of 
possible actions on each state. Call $\triangle$ the set of terminal states (where the game is over). 
Assumption \ref{tri} is satisfied as soon as from any state, there exists a policy under which is 
game ends in finite time. This is the case for almost all real games.
\black

\vip

We will check the following.

\begin{rk}\label{any} 
Grant Assumption \ref{tri} 
\vip
(i) Consider two SM policies $\pi,\pi'$ such that $\pi(x,a)=\pi'(x,a)$ for all $x\in\cX\setminus\{\triangle\}$,
all $a \in \cA_x$. Then $V_{\pi}(x)=V_{\pi'}(x)$ for all $x\in\cX$.

\vip
(ii) Consider any \emph{positive} SM policy $\pi$, i.e. such that
$\pi(x,a)>0$ for all $(x,a) \in \cZ$. Then
$\PP_{x,\pi}(T_\triangle <\infty)=1$ for all $x\in \cX$, where $T_\triangle=\inf\{t\geq 0 : X_t \in\triangle\}$.
\end{rk}

\blue In this case, Algorithm 2 may be rewritten as Algorithm 3 below, in the following sense. \black

\begin{rk}\label{plur}
Grant Assumption \ref{tri}. \blue Consider Algorithm 3,
using the same random elements as in Algorithm 2. \black
It holds that $\tpi^k(x,a)=\bpi^k(x,a)$ for all $k\geq 0$, all $x\in\cX\setminus\{\triangle\}$,
all $a \in \cA_x$. Hence $V_{\tpi^k}(x)=V_{\bpi^k}(x)$ for all $x \in\cX$ by Remark \ref{any}-(i).
\end{rk}

\begin{oframed}
\centerline{\blue Algorithm 3 (First visit algorithm with finite episodes) \black}
\vip\vip
\noindent We consider $\theta>0$ and a law $\mu_0$ on $\cX$. We set $\tQ^0(x,a)=0$ for all $(x,a)\in\cZ$.
For $k\geq 0$, assume that $k$ first finite episodes
$$
(X^i_0,A^i_0,X^i_1,R^i_1,A^i_1,\dots,X^i_{T^i_{\triangle-1}},R^i_{T^i_{\triangle-1}},A^i_{T^i_{\triangle-1}},X^i_{T^i_{\triangle}}), 
\quad i=1,\dots,k,
$$
have been built, with $T^i_\triangle=\inf\{t\geq 0 : X^i_t \in\triangle\}$, as 
well as the function $\tQ^{k}:\cZ\to\rr$. 
For $(x,a)\in\cZ$, we set
\begin{gather}\label{fvf0}
\tN_k(x,a)=\sum_{i=1}^k \indiq_{\{\ttau^i_{x,a}<\infty\}} \quad \hbox{and} \qquad \tN_k(x)=\sum_{a \in \cA_x} \tN_k(x,a)
\quad \hbox{(convention: $\sum_{i=1}^0=0$),}\\
\hbox{where} \qquad \ttau^i_{x,a}=\inf\{t\in \ig 0,T^i_\triangle\id : (X^i_t,A^i_t)=(x,a)\}\qquad \hbox{(convention: 
$\inf \emptyset=\infty$).}\notag
\end{gather}
Consider the SM policy defined, for $(x,a)\in\cZ$, by
\begin{equation}\label{fvf1}
\tpi^{k} (x,a)= \frac{\te_k(x)}{|\cA_x|} + \indiq_{\{a \in \argmax \tQ^k(x,\cdot)\}} 
\frac{1-\te_k(x)}
{|\argmax \tQ^k(x,\cdot)|},\; \hbox{where}\; \te_k(x)=\frac1{(1+\tN_k(x))^{\theta}}.
\end{equation}
We then build the $(k+1)$-th episode \blue
$$
(X^{k+1}_0,A^{k+1}_0,X^{k+1}_1,R^{k+1}_1,A^{k+1}_1,\dots,X^{k+1}_{T^{k+1}_{\triangle}-1},
R^{k+1}_{T^{k+1}_{\triangle}-1},A^{k+1}_{T^{k+1}_{\triangle}-1},X^{k+1}_{T^{k+1}_{\triangle}},R^{k+1}_{T^{k+1}_{\triangle}})
$$
\black using $X^{k+1}_0\sim \mu_0$ and the policy $\tpi^{k}$, as in Subsection \ref{mod}.
We then compute, for each $(x,a)\in \cZ$,
\begin{gather}\label{fvf2}
\tQ^{k+1}(x,a)=\frac{\sum_{i=1}^{k+1} \indiq_{\{\ttau^i_{x,a}<\infty\}}\tG^i_{x,a}}
{\tN_{k+1}(x,a)} \qquad \hbox{(convention: $\frac00=0$)},\\
\label{fvf3}\hbox{where} \qquad \tG^{i}_{x,a}= \blue \sum_{t =0}^{T^i_\triangle-\ttau^i_{x,a}-1} 
\black \gamma^t R^{i}_{\ttau^{i}_{x,a}+t+1}
\qquad \hbox{(convention : $\tG^{i}_{x,a}=0$ when $\ttau^i_{x,a}=\infty$).}
\end{gather}
\end{oframed}

\blue

Here is a pseudo-code for Algorithm 3:
here $Q(x,a)$ (resp $n(x,a)$, $N(x)$, $\pi(x,a)$, $\e(x)$) represents $\tilde Q^k(x,a)$ 
(resp. $\tilde N_k(x,a)$, $\tilde N_k(x)$, $\tilde \pi^k(x,a)$, $\tilde \e_k(x)$).

\begin{oframed}
\centerline{Pseudo-code for Algorithm 3}
\vip\vip

{\tt
set $k=0$
\vip
for all $(x,a) \in\cZ$, set $Q(x,a)=0$, $n(x,a)=0$, $N(x)=0$
\vip
do \hskip0.35cm for all $x \in \cX$, set $\e(x)=(1+N(x))^{-\theta}$
\vip
...  for all $(x,a) \in\cZ$, set 
$\pi(x,a)=\frac{\e(x)}{|\cA_x|} + \indiq_{\{a \in \argmax Q(x,\cdot)\}} \frac{1-\e(x)}{|\argmax Q(x,\cdot)|}$ 
\vip
... simulate $X_0,A_0,X_1,R_1,A_1,X_2,R_2,A_2,\dots,X_{T-1},R_{T-1},A_{T-1},X_T,R_T$ with

\hskip0.95cm  $X_0\sim\mu_0$ and using the policy $\pi$ until $X$ reaches $\triangle$ (at time $T$).
\vip
... for all $(x,a)\in \cZ$, if there is $t\in \ig 0,T \id$ such that $(X_t,A_t)=(x,a)$
\vip
... ... find $\tau=\min\{t\in \ig 0,T \id : (X_t,A_t)=(x,a)\}$
\vip
... ... set $G=\sum_{t=0}^{T-\tau-1}\gamma^t R_{\tau+t+1}$ and $Q(x,a)=(n(x,a)Q(x,a)+G)/(n(x,a)+1)$
\vip
... ... set $n(x,a)=n(x,a)+1$ and $N(x)=N(x)+1$
\vip
... set $k=k+1$
\vip
until $k$ has reached the desired value
}
\end{oframed}

\black

\section{Main contraction result}\label{cont}

The following contraction estimate is the key of our study.

\begin{lem}\label{contr}
Assume that $\gamma \in [0,1)$.
For a function $\e:\cX \to [0,1]$, for a function $q:\cZ\to \rr$
and for $x\in\cX$, we define the probability measure $\pi^\e_q(x,\cdot)$ on $\cA_x$ by
$$
\pi^\e_q(x,a)=\frac{\e(x)}{|\cA_x|}+ \indiq_{\{a \in \argmax q(x,\cdot)\}} 
\frac{1-\e(x)}{|\argmax q(x,\cdot)|}.
$$
We also introduce the function $\cH(\e,q):\cZ\to \rr$ defined by $\cH(\e,q)(x,a)=Q_{\pi^\e_q}(x,a)$,
recall \eqref{qpi}.
It holds that $\cH(\e,q)(x,a)\leq Q^*(x,a)$ for all $(x,a)\in\cZ$ and
$$
||\cH(\e,q)-Q^*||_\infty \leq \frac{\gamma}{1-\gamma}\Big( ||Q^*-q||_\infty + ||(q-Q^*)_+||_\infty 
+ 2 ||Q^*||_\infty ||\e||_\infty\Big),
$$
where $(q-Q^*)_+(x,a)=\max\{q(x,a)-Q^*(x,a),0\}$ and where $||\cdot||_\infty$ 
stand for $\max$ norms.
\end{lem}

\blue
As we will see in Section \ref{convgen}, $\hQ^{k}$ defined in Algorithm 1 satisfies a recursive equation 
of the shape
$$ \hQ^{k+1} = \hQ^{k} + \lambda_k(\cH(\epsilon_k, \hQ^k)-\hQ^k) + \text{small noise}.$$
To show that $\hQ^k$ converges to $Q^*$, it is natural to try to prove that $q\mapsto \cH(\epsilon, q)$ 
is an approximate contraction around $Q^*$. In this spirit, observe that we have $\cH(0,Q^*)=Q_{\pi^0_{Q^*}}=Q^*$, 
because
$V_{\pi^0_{Q^*}}=V^*$ by Theorem \ref{known}-(i), so that $Q_{\pi^0_{Q^*}}=Q^*$ by Theorem \ref{known}-(ii)-(iii).
\vip
For $q\mapsto\cH(\epsilon,q)$ to be an approximate contraction, we need $\gamma<1/2$ to have 
$\gamma/(1-\gamma)<1$.
\black

\begin{proof}
First, know from Theorem \ref{known}-(ii) that $\cH(\e,q)(x,a)=Q_{\pi^\e_q}(x,a)\leq Q^*(x,a)$
for all $(x,a) \in \cZ$. 
Next, by Theorem \ref{known}-(ii)-(iii), we have, for $(x,a)\in\cZ$, 
\begin{align*}
0\leq Q^*(x,a)-Q_{\pi^\e_q}(x,a)=&
\gamma \sum_{y \in \cX} P(x,a,y)\Big[\max_{b \in \cA_y} Q^*(y,b)-\sum_{b\in\cA_y} \pi^\e_q(y,b)Q_{\pi^\e_q}(y,b)
\Big]\\
=&\gamma \sum_{y\in \cX} P(x,a,y)\Big[\max_{b \in \cA_y} Q^*(y,b)- \sum_{b\in \cA_y} \pi^\e_q(y, b)Q^*(y,b)\Big]\\
&+ \gamma \sum_{y\in \cX} P(x,a,y)\sum_{b\in \cA_y} \pi^\e_q(y, b)\Big[Q^*(y,b)-Q_{\pi^\e_q}(y,b)\Big]\\
\leq & \gamma \sum_{y\in\cX}P(x,a,y)\Delta(y) + \gamma ||Q^* -  Q_{\pi^\e_q}||_\infty,
\end{align*}
where
$$
\Delta(y):=\max_{b \in \cA_y} Q^*(y,b)-\sum_{b \in \cA_y} \pi^\e_q(y,b)Q^*(y,b).
$$
We write $\Delta(y)=\Delta_1(y)+\Delta_2(y)+\Delta_3(y)+\Delta_4(y)$,
with
\begin{align*}
&\Delta_1(y)= \max_{b \in \cA_y} Q^*(y,b)- \max_{b \in \cA_y} q(y,b),\quad
&\Delta_2(y)=& \max_{b \in \cA_y} q(y,b)- \sum_{b\in\cA_x} q(y,b)\pi_q(y,b),\\
&\Delta_3(y)= \sum_{b\in\cA_y} [q(y,b)-Q^*(y,b)]\pi_q(y,b),\quad
&\Delta_4(y)=& \sum_{b\in\cA_y} Q^*(y,b) [\pi_q(y,b)-\pi_q^\e(y,b)],
\end{align*}
where we have set 
$$
\pi_q(x,a)=\indiq_{\{a \in \argmax q(x,\cdot)\}} \frac{1}{|\argmax q(x,\cdot)|}\qquad \hbox{for all $(x,a)\in\cZ$.}
$$
We clearly have $\Delta_1(y) \leq ||Q^*-q||_\infty$,
as well as $\Delta_3(y) \leq ||(q-Q^*)_+||_\infty$.
By definition of $\pi_q$, it holds that $\Delta_2(y)=0$. Finally,
$$
\Delta_4(y) \leq ||Q^*||_\infty \blue \sup_{y \in \cX}\sum_{b \in \cA_y}|\pi_q(y,b)-\pi_q^\e(y,b)|  \black 
\leq 2||Q^*||_\infty ||\e||_\infty.
$$
Thus
$$
\Delta(y) \leq ||Q^*-q||_\infty+||(q-Q^*)_+||_\infty+2||Q^*||_\infty ||\e||_\infty, 
$$
whence
$$ 
\sum_{y\in\cX}P(x,a,y)\Delta(y)\leq ||Q^*-q||_\infty+||(q-Q^*)_+||_\infty+2||Q^*||_\infty ||\e||_\infty.
$$
All in all, we have proved that
$$
0\leq Q^*(x,a)-Q_{\pi^\e_q}(x,a) \leq \gamma\Big(||Q^*-q||_\infty+||(q-Q^*)_+||_\infty+2||Q^*||_\infty ||\e||_\infty\Big)
+\gamma ||Q^* -  Q_{\pi^\e_q}||_\infty.
$$
We end with
$$
||Q^* -  Q_{\pi^\e_q}||_\infty \leq \gamma\Big(||Q^*-q||_\infty+||(q-Q^*)_+||_\infty+2||Q^*||_\infty ||\e||_\infty\Big)
+\gamma ||Q^* -  Q_{\pi^\e_q}||_\infty,
$$
from which the result readily follows, recalling that $\cH(\e,q)=Q_{\pi^\e_q}$.
\end{proof}

\section{A general convergence result}\label{agp}

We consider a finite set $\cY$ and denote by $\bF(\cY,[0,1])$
the set of all functions from $\cY$ to $[0,1]$ and by $\bF(\cY,\rr)$
the set of all functions from $\cY$ to $\rr$.

\begin{assumption}\label{cgt}
There exists $f_* \in \bF(\cY,\rr)$, $\rho \in (0,1)$ and $\beta>0$ such that
the function 
$$
\cM : \bF(\cY,[0,1])\times\bF(\cY,\rr)\to \bF(\cY,\rr)
$$ 
satisfies, for all $f \in \bF(\cY,\rr)$,
all $\eta \in \bF(\cY,[0,1])$,
\begin{gather}
\cM(\eta,f)(y)\leq f_*(y) \qquad \hbox{ for all $y\in\cY$}, \label{ccstarstar}\\
||f_*-\cM(\eta,f) ||_\infty \leq \rho ||f_*-f||_\infty + \beta ||(f - f_*)_+||_\infty + \beta ||\eta||_\infty. 
\label{ccstar}
\end{gather}
\end{assumption}

Our main results will be deduced from the following general abstract result.
\blue
A result in the same spirit is stated by Singh-Jaakkola-Littman-Szepesv\'ari \cite[Lemma 1]{sjls}.
However, they deal with less complex filtrations and assume rather a condition like
$||f_*-\cM(\eta,f) ||_\infty \leq \rho ||f_*-f||_\infty  + \beta ||\eta||_\infty $.
\black

\begin{prop}\label{gprop}
Assume that $\cM$ satisfies Assumption \ref{cgt} and consider 
\vip
\noindent $\bullet$ a family $(\cI^k_y)_{k\geq 0,y\in\cY}$  of $\sigma$-fields such that $\cI^k_y\subset \cI^{k+1}_{y'}$ 
for all $k\geq 0$, all $y,y'\in\cY$,

\noindent $\bullet$ some random $\hat f_0 \in \bF(\cY,\rr)$ such that $\hat f_0(y)$ is $\cI^0_y$-measurable 
for all $y\in\cY$,

\noindent $\bullet$ for each $k\geq 0$, some random
$\eta_k \in \bF(\cY,[0,1])$ which is $\cI^{k}_{y}$-measurable for all $y\in\cY$.

\noindent $\bullet$ for each $k\geq 1$, some random $\xi_k \in \bF(\cY,\rr)$ such that $\xi_k(y)$ is 
$\cI^k_y$-measurable for all $y\in\cY$,

\noindent $\bullet$ for each $k\geq 0$, some random $\lambda_k \in \bF(\cY,[0,1])$ such that $\lambda_k(y)$ is 
$\cI^{k}_y$-measurable for all $y\in\cY$.

\vip
Assume moreover that a.s., for all $y\in \cY$, we have 
$$
\lim_k \eta_k(y)=0,\qquad\sum_{k\geq 0} \lambda_k(y)=\infty,\qquad \sum_{k\geq 0} (\lambda_k(y))^2<\infty
$$ 
and that there is a constant $C>0$ such that for all $k\geq 0$,  all $y \in \cY$, 
$$
\E[\xi_{k+1}(y)|\cI^k_y]=0 \qquad\hbox{and}\qquad\E[(\xi_{k+1}(y))^2|\cI^k_y]\leq C.
$$

Define recursively, for $k\geq 0$,
$$
\hat f_{k+1}(y)=\hat f_k(y)
+\lambda_k(y)\Big(\cM(\eta_k,\hat f_k)(y)-\hat f_k(y)+\xi_{k+1}(y)\Big)\qquad \hbox{for all $y\in\cY$}.
$$
Then a.s., for all $y \in \cY$, $\lim_{k\to \infty} \hat f_k(y)=f_*(y)$.
\end{prop}

\blue
To prove the convergence of Algorithm 1, we will apply Proposition \ref{gprop} with
$\cM=\cH$ (and $\eta=\epsilon$ and $f=q$). As seen in Lemma \ref{contr}, $\cH$ satisfies \eqref{ccstar}
with $\rho=\gamma/(1-\gamma)$ and $\beta=\gamma/(1-\gamma)$ (of course, the value of $\beta$ in 
front of $\|\eta\|_\infty$ is not important). If assuming that $\rho+\beta<1$ (that is $\gamma<1/3$ 
in the context of Algorithm 1), we could more or less apply
\cite[Lemma 1]{sjls}, up to some filtration issues. However, taking advantage of the fact $\cH$ also
satisfies \eqref{ccstarstar}, we manage to conclude as soon as $\rho<1$ ($\gamma<1/2$ in the context of 
Algorithm 1), even if $\beta$ is very large.
Finally, we really need to introduce some state-dependent filtrations, to apply Lemma \ref{rm} correctly to 
Algorithm 1.
\black

\vip

The proof of this result relies on the following simple version of the Robbins-Monro theorem \cite{rm}.
See also its (short) proof in Appendix \ref{qpkr}.

\begin{lem}\label{rm}
Consider a filtration $(\cG_k)_{k\geq 0}$, a $\cG_0$-measurable real random variable $Z_0$, as well as some
$(\cG_k)_{k\geq 0}$-adapted sequences $(\theta_k)_{k\geq 0}$ and $(\zeta_k)_{k\geq 1}$ of real random variables,
with $\theta_k$ valued in $[0,1]$. Assume that a.s., $\sum_{k\geq 0} \theta_k =\infty$ and $\sum_{k\geq 0} \theta_k^2 
<\infty$ and that there is a constant $C>0$ such that for all $k\geq 0$,
$\E[\zeta_{k+1}|\cG_k]=0$ and $\E[(\zeta_{k+1})^2|\cG_k]\leq C$. Define recursively, for $k\geq 0$,
$$
Z_{k+1}=(1-\theta_k)Z_k+\theta_k\zeta_{k+1}.
$$
Almost surely, $\lim_{k}Z_k=0$.
\end{lem}

\begin{proof}[Proof of Proposition \ref{gprop}]
First, one easily checks by induction that for all $k\geq 0$, all $y\in\cY$, \blue $\hf_k(y)$ \black 
is $\cI^k_y$-measurable.
Our goal is to show that a.s., for all $y\in\cY$,
$\lim_k \hphi_k(y)=0$, where $\hphi_k(y)=\hf_k(y)-f_*(y)$. For all $k\geq 0$, all $y \in \cY$, we have
\begin{align}
\hphi_{k+1}(y)=&\hphi_{k}(y)+\lambda_k(y)[\cM(\eta_k,\hat f_k)(y)-f_*(y)-\hphi_k(y)+\xi_{k+1}(y)] \notag\\
=&(1-\lambda_k(y))\hphi_{k}(y)+\lambda_k(y)[\cM(\eta_k,\hat f_k)(y)-f_*(y)]+\lambda_k(y)\xi_{k+1}(y).\label{tbu}
\end{align}
We now divide the proof into several steps.
\vip

{\it Step 1.} We define $(W_k(y))_{k\geq 0,y\in\cY}$ by $W_0(y)=\hphi_{0}(y)$ and, for $k\geq 0$,
$$
W_{k+1}(y)= (1-\lambda_k(y))W_{k}(y)+\lambda_k(y)\xi_{k+1}(y). 
$$
For each fixed $y \in\cY$, we can apply Lemma \ref{rm} with $\cG_k=\cI^k_y$, $Z_0=W_0(y)$, 
$\theta_k=\lambda_k(y)$, $\zeta_k=\xi_k(y)$, so that $Z_k=W_k(y)$.
We get that a.s., for each $y\in\cY$, a.s., $\lim_k  W_{k}(y)=0$.

\vip

{\it Step 2.} We now show by induction that for all $k\geq 0$, all $y\in \cY$, $\hphi_k(y)\leq W_k(y)$. 
This will imply, thanks to Step 1, that \blue $\lim_k ||(\hphi_k)_+||_\infty=0$ a.s. \black

\vip
First, $\hphi_0(y)= W_0(y)$ for all $y\in \cY$. If next $\hphi_k(y)\leq W_k(y)$ for some $k\geq 0$, then
using \eqref{tbu} and Assumption \ref{cgt},
$$
\hphi_{k+1}(y)\leq (1-\lambda_k(y))\hphi_{k}(y)+\lambda_k(y)\xi_{k+1}(y) \leq (1-\lambda_k(y))W_{k}(y)
+\lambda_k(y)\xi_{k+1}(y)=W_{k+1}(y).
$$

{\it Step 3.} We prove here that a.s., $\sup_{k\geq 0}||\hphi_k||_\infty<\infty$. 
\vip

To this end, 
it suffices to prove that $\sup_{k\geq 0}||\Delta_k||_\infty<\infty$, where $\Delta_k(y)=W_k(y)-\hphi_k(y)$,
because $\sup_{k\geq 0}||W_k||_\infty<\infty$ a.s., since $\lim_k  ||W_{k}||_\infty=0$ a.s. by Step 1.
We recall that $\Delta_k(y)\geq0$ by Step 2 and write
\begin{align}
\Delta_{k+1}(y)= & (1-\lambda_k(y))\Delta_{k}(y)+\lambda_k(y)|f_*(y)-\cM(\eta_k,\hat f_k)(y)| \notag\\
\leq& (1-\lambda_k(y)) \Delta_{k}(y) +\lambda_k(y)[\rho||\hphi_k||_\infty+\beta ||(\hphi_k)_+||_\infty 
+\beta ||\eta_k||_\infty]\label{ttroc}
\end{align}
by Assumption \ref{cgt}. Since $||\hphi_k||_\infty\leq ||\Delta_k||_\infty+||W_k||_\infty$, this gives
$$
\Delta_{k+1}(y)\leq (1-\lambda_k(y))||\Delta_k||_\infty+\lambda_k(y)\rho||\Delta_k||_\infty + \lambda_k(y)H_k
$$
where $H_k=\rho||W_k||_\infty+\beta ||(\hphi_k)_+||_\infty +\beta ||\eta_k||_\infty$.
Setting $K=\sup_{k\geq 0}H_k$, which is a.s. finite by Steps 1 and 2, we end with
$$
\Delta_{k+1}(y)\leq (1-(1-\rho)\lambda_k(y))||\Delta_k||_\infty + \lambda_k(y)K.
$$

If $||\Delta_k||_\infty\geq K/(1-\rho)$, we conclude that $|\Delta_{k+1}(y)|\leq ||\Delta_k||_\infty$.
\vip
If  $||\Delta_k||_\infty< K/(1-\rho)$, we conclude that $|\Delta_{k+1}(y)|\leq K/(1-\rho)$.
\vip
We thus always have $||\Delta_{k+1}||_\infty\leq \max\{||\Delta_k||_\infty, K/(1-\rho)\}$,
from which we easily conclude that for all $k\geq 0$, 
$||\Delta_{k}||_\infty\leq \max\{||\Delta_0||_\infty, K/(1-\rho)\}$.
\vip

{\it Step 4.} We fix $k_0\geq 1$ and set 
\begin{equation}\label{dkz}
D_{k_0}=||W_{k_0}||_\infty+\sup_{k\geq k_0} \Big[||\hphi_k||_\infty
+\frac\beta\rho\Big(||(\hphi_k)_+||_\infty+||\eta_k||_\infty\Big)\Big].
\end{equation}
For each $y \in \cY$ fixed, we define $(Y_k^{k_0}(y))_{k\geq k_0}$ by $Y_{k_0}^{k_0}(y)=D_{k_0}$ 
and by induction, for $k\geq k_0$, 
$$
Y^{k_0}_{k+1}(y)=(1-\lambda_k(y))Y^{k_0}_{k}(y)+\rho \lambda_k(y)D_{k_0}.
$$ 
We show here that $\lim_k Y_k^{k_0}(y)=\rho D_{k_0}$ a.s. 

\vip
It suffices to note that for all $k\geq k_0$, 
$(Y^{k_0}_{k+1}(y)-\rho D_{k_0})=(1-\lambda_k(y))(Y^{k_0}_{k}(y)-\rho D_{k_0})$, whence
$Y^{k_0}_{k+1}(y)-\rho D_{k_0}=(1-\rho)D_{k_0}\prod_{\ell=k_0}^{k} (1-\lambda_\ell(y))$. This last quantity a.s.
tends to $0$ as $k\to \infty$ because by assumption, $\sum_{\ell \geq k_0} \lambda_\ell(y)=\infty$ a.s.

\vip
{\it Step 5.} Here we check that $\Delta_k(y)\leq Y^{k_0}_k(y)$
for all $k\geq k_0\geq0$ and all $y \in \cY$. 

\vip
We obviously have $\Delta_{k_0}(y)\leq ||W_{k_0}||_\infty
+||\hphi_{k_0}||_\infty\leq D_{k_0}=Y^{k_0}_{k_0}(y)$ and, assuming by induction that $\Delta_k(y)\leq Y^{k_0}_k(y)$
for some $k\geq k_0$, we write, recalling \eqref{ttroc},
\begin{align*}
\Delta_{k+1}(y)\leq & (1-\lambda_k(y))Y^{k_0}_{k}(y)
+\lambda_k(y)[\rho ||\hphi_k||_\infty+\beta||(\hphi_k)_+||_\infty+\beta||\eta_k||_\infty]\\
=&(1-\lambda_k(y))Y^{k_0}_{k}(y)
+\rho \lambda_k(y)\Big[||\hphi_k||_\infty+\frac\beta\rho||(\hphi_k)_+||_\infty+\blue \frac \beta\rho \black
||\eta_k||_\infty\Big]\\
\leq & (1-\lambda_k(y))Y^{k_0}_{k}(y)+\rho \lambda_k(y)D_{k_0} \\
=& Y^{k_0}_{k+1}(y).
\end{align*}

{\it Step 6.} Since $\lim_k ||W_k||_\infty =0$ a.s. by Step 1, we conclude that a.s., 
$\limsup_k ||\hphi_k||_\infty =\limsup_k ||\Delta_k||_\infty$. Hence by Steps 5 and 4
(and since $\Delta_k(y)\geq 0$, see Step 3), for any $k_0\geq 0$, a.s.,
$$
\limsup_k ||\hphi_k||_\infty \leq \limsup_k ||Y^{k_0}_k||_\infty=\rho D_{k_0}.
$$
As a consequence, we a.s. have
$$
\limsup_k ||\hphi_k||_\infty \leq \rho\limsup_{k_0} D_{k_0} .
$$
Recalling the definition \eqref{dkz} of $D_{k_0}$ and that 
$\lim_k ||W_k||_\infty=0$ by Step 1, that $\lim_k ||(\hphi_k)_+||_\infty=0$ by Step 2 and 
that $\lim_k ||\eta_k||_\infty=0$ by assumption, we end with
$$
\limsup_k ||\hphi_k||_\infty \leq \rho \limsup_k ||\hphi_k||_\infty.
$$
Since finally $\limsup_k ||\hphi_k||_\infty <\infty$ by Step 3 and since $\rho \in (0,1)$, 
we find that $\limsup_k ||\hphi_k||_\infty =0$ a.s., which was our goal.
\end{proof}

\section{Convergence of \blue Algorithm 1 \black}\label{convgen}

\blue Here apply the general result of the previous section to write down the \black

\begin{proof}[Proof of Theorem \ref{mrg}] We divide the proof in two steps.
\vip

{\it Step 1.} We claim that it suffices to show that $\lim_k\hQ^k(x,a)=Q^*(x,a)$ a.s. for all $(x,a)\in \cZ$. 

\vip
Indeed, assume this is the case. Observe that by construction, see \eqref{gg1} and Lemma \ref{contr},
$$
\hpi^k=\pi^{\e_k}_{\hQ^k},  \qquad \hbox{so that} \qquad  Q_{\hpi^k}=\cH(\e_k,\hQ^k).
$$
By Lemma \ref{contr}, 
the facts that $\lim_k \e_k(x)=0$ by assumption and that 
$\lim_k\hQ^k(x,a)=Q^*(x,a)$ for all $(x,a)\in\cZ$ imply that
$\lim_k Q_{\hpi^k}(x,a)=Q^*(x,a)$ for all  $(x,a)\in\cZ$. But Theorem \ref{known}-(ii) tells us that
for all $x\in \cX$,
$$
V_{\hpi^k}(x)=\sum_{a\in\cA_x} Q_{\hpi^k}(x,a)\hpi^k(x,a)=\frac{\e_k(x)}{|\cA_x|}\sum_{a\in\cA_x} Q_{\hpi^k}(x,a)
+ (1-\e_k(x)) \max_{a \in\cA_x} Q_{\hpi^k}(x,a)
$$
by definition of $\hpi^k$, see \eqref{gg1} again. 
Using that $\lim_k \e_k(x)=0$ and $\lim_k Q_{\hpi^k}(x,a)=Q^*(x,a)$, we conclude that
$$
\lim_kV_{\hpi^k}(x)=  \max_{a \in\cA_x} Q^*(x,a),
$$
which equals $V^*(x)$ by Theorem \ref{known}-(iii).
\vip

{\it Step 2.} To prove that $\lim_k\hQ^k(x,a)=Q^*(x,a)$ a.s. for all $(x,a)\in \cZ$, we aim to 
apply Proposition \ref{gprop} with $\cY=\cZ$, with $\hf^k=\hQ^k$ for all $k\geq 0$, with the map, 
$\cM(\e,q)= \cH(\e,q)=Q_{\pi_q^\e}$ for all $\e \in \bF(\cZ,[0,1])$ and all $q \in \bF(\cZ,\rr)$, 
with $\eta_k(z)=\e_k(x)$ and $\lambda_k(z)=\indiq_{\{\tau^{k+1}_{x,a}<\infty\}}\alpha_k(x,a)$ 
for all $k\geq 0$, all $z=(x,a)\in \cZ$, with
$$
\xi_{k}(z)=\indiq_{\{\tau^{k}_{x,a}<\infty\}} \Big[G^{k}_{x,a}-\cH(\e_{k-1},\hQ^{k-1})(x,a)\Big] 
\qquad \hbox{for all $k\geq 1$, all $z=(x,a) \in \cZ$,}
$$
and with the family of filtrations $(\cI^k_z)_{k\geq 0,z \in \cZ}$ defined by (for 
$z=(x,a)$)
$$
\cI^k_{z}= \cF^k \lor \cG^{k+1}_{\tau^{k+1}_{x,a}},
$$
where $\cG^k_\ell = \sigma(X_0^k,A_0^k,R_1^k,X_1^k,A_1^k,\dots,R_\ell^k,X_\ell^k,A_\ell^k)$
and $\cF^k=\lor_{i=1}^k\cG^i_\infty$ (and $\cF^0=\{\emptyset, \Omega\}$).

\vip

First, Lemma \ref{contr} tells us that $\cH$ satisfies Assumption \ref{cgt} with $f_*=Q^*$, with 
$\rho=\gamma/(1-\gamma) \in (0,1)$ (recall that $\gamma \in (0,1/2)$) 
and $\beta=[\gamma/(1-\gamma)]\max\{1,2||Q^*||_\infty\}$.

\vip

We obviously have $\cI^k_{z}\subset \cI^{k+1}_{z'}$ for all $k\geq 0$, all $z,z' \in\cZ$. Moreover,
$\hQ^0$ is deterministic by assumption and thus $\hQ^0(z)$ is 
$\cI^0_z$-measurable for all $z \in \cZ$. For each $k\geq 0$, each $z=(x,a)\in\cZ$, $\eta_k(z)=\e_k(x)$
is $\cF^k$-measurable by assumption and thus \blue $\cI^{k}_{z}$-measurable. \black
For each $k\geq 1$, each $z=(x,a) \in \cZ$,
$\xi_k(z)=\indiq_{\{\tau^{k}_{x,a}<\infty\}} [G^{k}_{x,a}-\cH(\e_{k-1},\hQ^{k-1})(x,a)]$ is $\cF^k$-measurable and thus
$\cI^k_z$-measurable. Finally, for each $k\geq 0$, each $z=(x,a) \in \cZ$,
$\lambda_k(z)=\indiq_{\{\tau^{k+1}_{x,a}<\infty\}}\alpha_k(x,a)$ is $\cI^k_z$-measurable,
since $\alpha_k(x,a)$ is \blue $\cI^k_z$-measurable \black by assumption.

\vip

We next observe that for all $k\geq 0$, all $z=(x,a) \in\cZ$, recalling \eqref{gg3},
\begin{align*}
\hQ^{k+1}(z)=&(1-\lambda_k(z))\hQ^k(z)+\lambda_k(z)G^{k+1}_{x,a}\\
=&\hQ^k(z)+\lambda_k(z)\Big(\cH(\e_{k},\hQ^{k})(z)-\hQ^k(z)  +[G^{k+1}_{x,a}-\cH(\e_{k},\hQ^{k})(z)] \Big)\\
=&\hQ^k(z)+\lambda_k(z)\Big(\cH(\e_{k},\hQ^{k})(z)-\hQ^k(z)  +\indiq_{\{\tau^{k+1}_{x,a}<\infty\}}
[G^{k+1}_{x,a}-\cH(\e_{k},\hQ^{k})(z)] \Big)
\end{align*}
because $\lambda_k(z)=0$ when $\tau^{k+1}_{x,a}=\infty$. Hence
$$
\hQ^{k+1}(z)=\hQ^k(z)+\lambda_k(z)\Big(\cM(\eta_{k},\hQ^{k})(z)-\hQ^k(z)  +\xi_{k+1}(z)\Big)
$$
as desired.
\vip
For $k\geq 0$ and $z=(x,a)\in\cZ$, it holds that
$$
\E[\xi_{k+1}(z) | \cI^k_z]=\indiq_{\{\tau^{k+1}_{x,a}<\infty\}}\Big(\E[G^{k+1}_{x,a}|\cI^k_z] -\cH(\e_{k},\hQ^{k})(z) \Big)
=0.
$$
Indeed, on $\{\tau^{k+1}_{x,a}<\infty\}$, we have $\E[G^{k+1}_{x,a}|\cI^k_z]=Q_{\hpi^{k}}(x,a)$ (i.e.
$\E[G^{k+1}_{x,a}|\cI^k_z]=\cH(\e_{k},\hQ^{k})(x,a)$, see Step 1): recalling \eqref{gg2},
\begin{align*}
\E[G^{k+1}_{x,a}|\cI^k_z]=&\E\Big[R^{k+1}_{\tau^{k+1}_{x,a}+1}\Big|\cI^k_{x,a}\Big] 
+ \gamma \E\Big[\sum_{t\in \nn} \gamma^t R^{k+1}_{\tau^{k+1}_{x,a}+2} \Big|\cI^k_{x,a}\Big] \\
=& \sum_{y\in\cX}P(x,a,y) g(x,a,y) + \gamma \E\Big[\E_{X_{\tau^{k+1}_{x,a}+1},\hpi^k}[G]\Big|\cI^k_{x,a}\Big] 
\end{align*}
by the strong Markov property, recall \eqref{ttt} and \eqref{G}. Recalling that $V_\pi(x)=E_{x,\pi}[G]$,
\eqref{qstar1} and \eqref{qpi}, this gives
$$
\E[G^{k+1}_{x,a}|\cI^k_z] = r(x,a)+ \gamma\sum_{y \in \cX}P(x,a,y) V_{\hpi^k}(y) = Q_{\hpi^k}(x,a)
$$
as desired.
\vip

Moreover,
$$
\E[(\xi_{k+1}(z))^2 | \cI^k_z]=\indiq_{\{\tau^{k+1}_{x,a}<\infty\}}{\rm Var}[G^{k+1}_{x,a}|\cI^k_z] \leq 
\indiq_{\{\tau^{k+1}_{x,a}<\infty\}}\E[(G^{k+1}_{x,a})^2|\cI^k_z].
$$
As previously, using that $(u+v)^2\leq2u^2+2v^2$, we find that on $\{\tau^{k+1}_{x,a}<\infty\}$,
$$
\E[(\xi_{k+1}(z))^2 | \cI^k_z]\leq 2\sum_{y\in\cX}P(x,a,y) \int_\rr z^2 S(x,a,y,\dd z) 
+ 2 \gamma^2 \E\Big[\E_{X_{\tau^{k+1}_{x,a}+1},\hpi^k}[G^2]\Big|\cI^k_{x,a}\Big].
$$
The first term is bounded by $2K$, where $K=\sup_{(x,a)\in\cZ,y\in\cX} \int_\rr z^2 S(x,a,y,\dd z)$ 
(see Setting \ref{set}).
The second term is bounded, because $\sup_{x\in \cX, \pi} \E_{x,\pi}[G^2] <\infty$: thanks to the Minkowski inequality,
$$
\E_{x,\pi}[G^2] \leq \Big( \sum_{t\in\nn} \gamma^t \E_{x,\pi}[R_{t+1}^2]^{1/2} \Big)^2 \leq \frac K{(1-\gamma)^2}.
$$
All in all, we have proved that
$$
\E[(\xi_{k+1}(z))^2 | \cI^k_z]\leq 2 K + \frac {2 \gamma^2K}{(1-\gamma)^2}.
$$

By assumption, see \eqref{cqdev}, the three conditions 
$\lim_{k} \eta_k(z)=\lim_{k} \e_k(x)=0$,
$\sum_{k\geq 0} \lambda_k(z)=\sum_{k\geq 0}\indiq_{\{\tau^{k+1}_{x,a}<\infty\}}\alpha_k(x,a)=\infty$ 
and $\sum_{k\geq 0} (\lambda_k(z))^2=\sum_{k\geq 0}\indiq_{\{\tau^{k+1}_{x,a}<\infty\}}(\alpha_k(x,a))^2<\infty$ 
are a.s. fulfilled for all $z=(x,a)\in\cZ$.
\vip

By Proposition \ref{gprop}, we deduce that a.s., for all $z=(x,a)\in\cZ$, $\lim_k\hQ^k(x,a)=Q^*(x,a)$.
\end{proof}

\section{Convergence of \blue Algorithm 2 \black}\label{convfv}
We now handle the

\begin{proof}[Proof of Theorem \ref{mmrr}]
\blue Algorithm 2 is a particular case of Algorithm 1, \black namely when choosing
$\hQ^0=0$ (so that $\hpi^0(x,a)=|\cA_x|^{-1}$ for all $(x,a)\in\cZ$) and, for all $k\geq 0$ and $(x,a)\in\cZ$, 
$\nu_k(x,a)=\mu_0(x)\hpi^k(x,a)$, $\e_k(x)=(1+N_k(x))^{-\theta}$ and $\alpha_k(x,a)=(N_{k+1}(x,a))^{-1}$ 
\blue (convention: $1/0=0$), \black
where we recall that (convention: $\sum_{i=1}^0=0$)
$$
N_k(x,a)=\sum_{i=1}^k \indiq_{\{\tau^i_{x,a}<\infty\}} \qquad \hbox{and} \qquad N_k(x)=\sum_{a\in\cA_x} N_k(x,a).
$$
Observe that $\alpha_k(x,a)$ is indeed $\cF^k \lor \cG^{k+1}_{\tau^{k+1}_{x,a}}$-measurable.
Indeed, the only issue is to check that for all $k\geq 0$, all $(x,a)\in\cZ$, we have
$$
\bQ^{k+1}(x,a)=(1-\alpha_{k}(x,a)\indiq_{\{\tau^{k+1}_{x,a}<\infty\}})\bQ^{k}(x,a) + 
\alpha_{k}(x,a)\indiq_{\{\tau^{k+1}_{x,a}<\infty\}}G^{k+1}_{x,a},
$$
but this immediately follows from the definition \eqref{fv2} of $\bQ^{k+1}$ and the facts that 
$N_{k+1}(x,a)=N_k(x,a)+\indiq_{\{\tau^{k+1}_{x,a}<\infty\}}$
(when $k=0$, this also uses that $\bQ^0(x,a)=0$). 

\vip

To show that Theorem \ref{mrg} applies, we only have to verify \eqref{cqdev}.
Assume for a moment that for all $(x,a)\in\cZ$, $\lim_k N_k(x,a)=\infty$ a.s.
Then of course $\lim_k N_k(x)=\infty$, whence $\lim_k \e_k(x)=0$ a.s. Moreover,
$$
\sum_{k\geq 0}\alpha_k(x,a)\indiq_{\{\tau^{k+1}_{x,a}<\infty\}}
=\sum_{k\geq 1} \frac{\indiq_{\{\tau^{k}_{x,a}<\infty\}}}{N_{k}(x,a)}=\sum_{\ell \geq 1} \frac 1\ell = \infty,
$$
and 
$$
\sum_{k\geq 0}(\alpha_k(x,a))^2\indiq_{\{\tau^{k+1}_{x,a}<\infty\}}
=\sum_{k\geq 1} \frac{\indiq_{\{\tau^{k}_{x,a}<\infty\}}}{(N_{k}(x,a))^2}
=\sum_{\ell \geq 1} \frac 1 {\ell^2}<\infty.
$$

It thus only remains to show that for each $(x,a)\in\cZ$, that we now fix, $\lim_k N_k(x,a)=\infty$ a.s.
First, since the family $(X^i_0)_{i\geq 1}$ is i.i.d. and $\mu_0$ distributed with $\mu_0(x)>0$,
we conclude from the law of large numbers that a.s.,
\begin{equation}\label{ruru}
N_k(x)\geq \sum_{i=1}^k \indiq_{\{X^i_0=x\}} \sim \mu_0(x) k \qquad \hbox{as $k\to\infty$}.
\end{equation}
Next, since 
$$
N_k(x,a) \geq \sum_{i=1}^k \indiq_{\{X^i_0=x,A^i_0=a\}}=:S_k,
$$
it suffices to show that $\lim_k S_k=\infty$.
We introduce $\Lambda_k=\sum_{i=1}^k \PP(X^i_0=x,A^i_0=a | \cF^{i-1})$, so that $(M_k=S_k-\Lambda_k)_{k\geq 0}$
is a $(\cF^k)_{k\geq 0}$-martingale (with $S_0=\Lambda_0=0$).
We recall that by construction, see \eqref{fv1}, for all $i\geq 1$,
$$
\PP(X^i_0=x,A^i_0=a|\cF^{i-1})=\mu_0(x)\bpi^{i-1}(x,a)\geq \mu_0(x)\frac{\e_{i-1}(x)}{|\cA_{x}|}= 
\frac{\mu_0(x)}{|\cA_{x}|(1+N_{i-1}(x))^\theta}.
$$
Hence \eqref{ruru} implies that $\lim_k \Lambda_k=\infty$ a.s., because $\theta \in (0,1]$.
For \blue $A\in \nn$, \black we introduce the stopping time $\sigma_A=\inf\{k\geq 0 : S_k\geq A\}$. 
For all $k\geq 1$, We have $\E[M_{\sigma_A\land k}]=0$,
i.e. $\E[S_{\sigma_A\land k}]=\E[\Lambda_{\sigma_A\land k}]$. Hence $\E[\Lambda_{\sigma_A\land k}] \leq A$ whence,
by monotone convergence, $\E[\Lambda_{\sigma_A}] \leq A$. This implies that $\Lambda_{\sigma_A}<\infty$ a.s.,
whence $\sigma_A<\infty$ a.s. because $\lim_k \Lambda_k=\infty$. Since this holds true for all 
\blue $A\in \nn$, \black we conclude that $\lim_k S_k=\infty$ as desired.
\end{proof}

\section{The case where the episodes are finite}\label{mrs}

In this section, we grant Assumption \ref{tri}. We set 
$\cZ_\triangle=\{(x,a) : x \in \cX\setminus\triangle, a \in\cA_x\}$.

\begin{proof}[Proof of Remark \ref{any}] We start with (i).
For a given starting point $x \in \cX$ and two SM policies $\pi,\pi'$ coinciding on $\cZ_\triangle$,
we \blue claim that we can build simultaneously, by coupling, \black the corresponding processes 
$(X_0,A_0,X_1,R_1,A_1,\dots,X_{t},R_t,A_t,\dots )$ 
and $(X_0',A_0',X_1',R_1',A_1',\dots,X_{t}',R_t',A_t',\dots )$  such that
$T_\triangle=T'_\triangle$  and
\begin{align*}
(X_0,A_0,X_1,R_1,A_1,\dots,X_{T_\triangle-1},&R_{T_\triangle-1},A_{T_\triangle-1},X_{T_\triangle})\\
=&(X_0',A_0',X_1',R_1',A_1',\dots,X'_{T'_\triangle-1},R'_{T'_\triangle-1},A'_{T'_\triangle-1},X'_{T'_\triangle}),
\end{align*}
where $T_\triangle, T'_\triangle$ are the entrance times of $(X_t)_{t\geq 0}$ and 
$(X_t')_{t\geq 0}$ in $\triangle$. 

\vip

\blue Indeed, set $X_0=X_0'=x$. If $x\notin \Delta$ (else the claim is obvious), choose 
$A_0\sim\pi(x,\cdot)$  and $A_0'=A_0$ (which is relevant since $\pi'(x,\cdot)=\pi(x,\cdot)$),
choose $X_1\sim P(X_0,A_0,\cdot)$ and $X'_1=X_1$, choose $R_1\sim S(X_0,A_0,X_1,\cdot)$ and $R_1'=R_1$.
If $X_1\notin \Delta$ (else complete the construction arbitrarily), choose $A_1\sim\pi(X_1,\cdot)$  
and $A_1'=A_1$ (which is relevant since $\pi'(X_1,\cdot)=\pi(X_1,\cdot)$),
choose $X_2\sim P(X_1,A_1,\cdot)$ and $X'_2=X_2$, choose $R_2\sim S(X_1,A_1,X_2,\cdot)$ and $R_2'=R_2$. Etc.
\black

\vip

Moreover, $R_{t+1}=R_{t+1}'=0$ for all $t\geq T_\triangle=T'_\triangle$,
since the set $\triangle$ is absorbing and since $S(x,a,y,\dd z)=\delta_0(\dd z)$ as soon as $x\in \triangle$.
Thus $G=G'$, where
$$
G=\sum_{t\geq 0} \gamma^t R_{t+1}=
\sum_{t=0}^{T_\triangle-1}\gamma^t R_{t+1} \qquad \hbox{and} \qquad G'=\sum_{t\geq 0} \gamma^t R'_{t+1}
=\sum_{t=0}^{T'_\triangle-1}\gamma^t R'_{t+1},
$$
with the convention that $\sum_{t=0}^{-1}=0$.
Recalling \eqref{vpi}, we conclude that indeed, $V_\pi(x)=V_{\pi'}(x)$ for all $x\in \cX\setminus\triangle$.
Finally, $V_\pi(x)=V_{\pi'}(x)=0$ when $x\in\triangle$.

\vip

For (ii), consider a {\it positive} SM policy $\pi$. Recall that when using $\pi$, the process $(X_t)_{t\geq 0}$
is a Markov chain with transition matrix $P_\pi(x,y)=\sum_{a\in\cA_x}\pi(x,a)P(x,a,y)$.
Clearly, $P_\pi(x,y)>0$ if and only if \blue $\bar P(x,y)=\max_{a \in\cA_x} P(x,a,y)>0$. \black Hence we deduce from
Assumption \ref{tri} that for all $x\in\cX$, there are $y\in\triangle$ and $n\geq 0$
such that $P_\pi^n(x,y)>0$. The state space $\cX$ being finite, we classically conclude that indeed,
$T_\triangle=\inf\{t\geq 0 : X_t \in \triangle\}<\infty$ a.s. under $\PP_{x,\pi}$ for all $x\in\cX$.
\end{proof}

We next explain the

\begin{proof}[Proof of Remark \ref{plur}]
Here we use simultaneously \blue Algorithms 2 and 3, with the same random elements. \black
\vip
For all $(x,a) \in \cZ$,  we have $\tQ^0(x,a)=\bQ^0(x,a)=\tN_0(x,a)=N_0(x,a)=\tN_0(x)=N_0(x)=0$.

\vip

Hence  $\tpi^0(x,a)=\bpi^0(x,a)$.
Thus we can build the same first episode in both algorithms (stopped when reaching $\triangle$, at time 
$T^1_\triangle$, for the second one).
Since $\triangle$ is absorbing whatever the policy and since $R^1_{t}=0$ for all $t\geq T^1_\triangle$,
we conclude that for all $(x,a)\in\cZ_\triangle$, $\ttau^1_{x,a}=\tau^1_{x,a}$ and $\tG^1_{x,a}=G^1_{x,a}$,
whence $\tN_1(x,a)=N_1(x,a)$, $\tN_1(x)=N_1(x)$ and $\tQ^1(x,a)=\bQ^1(x,a)$. 
\vip

Thus $\tpi^1(x,a)=\bpi^1(x,a)$ for all $(x,a) \in \cZ_\triangle$, and we can build the same second episode 
in both algorithms (stopped when reaching $\triangle$, at time 
\blue $T^2_\triangle$, \black for the second one). Observe that this does not require that 
$\tpi^1(x,a)=\bpi^1(x,a)$ when 
$x \in \triangle$, because such values are never used in \blue Algorithm 3. \black
Since $\triangle$ is absorbing whatever the policy and since $R^2_{t}=0$ for all $t\geq T^2_\triangle$,
we conclude that for all $(x,a)\in\cZ_\triangle$, $\ttau^2_{x,a}=\tau^2_{x,a}$ and $\tG^2_{x,a}=G^2_{x,a}$,
whence $\tN_2(x,a)=N_2(x,a)$, $\tN_2(x)=N_2(x)$ and $\tQ^2(x,a)=\bQ^2(x,a)$. 
\vip

Iterating this argument, we find that  $\tpi^k(x,a)=\bpi^k(x,a)$ for all $(x,a) \in \cZ_\triangle$,
all $k\geq 0$.
\end{proof}

\section{Counter-example}\label{counterex}

\blue 
We first assume that $\gamma \in (1/2,1)\cap \Q$ and explain in Remark \ref{rkHH} how to treat 
the general case $\gamma\in (1/2,1)$.
\black

\begin{proof}[Proof of Proposition \ref{ptb} \blue when $\gamma \in \Q\cap(1/2,1)$ \black]
We recall that $\cX=\{\zz\}$, that $\cA_0=\{0,1\}$, that $P(\zz,0,\zz)=P(\zz,1,\zz)=1$, and that
$S(\zz,0,\zz,\cdot)=\delta_0$ and $S(\zz,1,\zz,\cdot)=\delta_1$. In other words, there is only one possible state
and two actions. When one chooses the action $0$, the (deterministic) reward is $0$, and when
one chooses the action $1$, the (deterministic) reward is $1$.
There are of course two extremal SM policies, that we denote by $\pi_0$ and $\pi_1$, and that are given by
$$
\pi_0(\zz,0)=1,\quad \pi_0(\zz,1)=0 \qquad \hbox{and} \qquad \pi_1(\zz,0)=0,\quad \pi_1(\zz,1)=1.
$$

Let us mention, although we will not use it, that using Theorem \ref{known} and \eqref{qpi},
with $r(\zz,0)=0$, $r(\zz,1)=1$ and $P_{\pi_0}(\zz,\zz)=P_{\pi_1}(\zz,\zz)=1$, 
one finds that $Q^*(\zz,0)=\gamma/(1-\gamma)$, $Q^*(\zz,1)=1/(1-\gamma)$, so that 
the optimal SM strategy is $\pi^*=\pi_1$ and that
\begin{gather*}
V_{\pi_0}(\zz)=0,\quad Q_{\pi_0}(\zz,0)=0,\quad Q_{\pi_0}(\zz,1)=1, \\
V_{\pi_1}(\zz)=\frac 1 {1-\gamma},\quad  Q_{\pi_1}(\zz,0)=\frac \gamma{1-\gamma},\quad 
Q_{\pi_0}(\zz,1)=\frac 1 {1-\gamma}.
\end{gather*}
For $Q:\cZ \to \rr$,  we will use the notation $Q=(u,v)$, where $u=Q(\zz,0)$ and
$v=Q(\zz,1)$.
\vip

We now design $\hQ^0=(u_0,v_0)$ and the families $(\nu_k)_{k\geq 0}$ and $(\alpha_k)_{k\geq 0}$
such that when applying \blue Algorithm 1 \black with $\e_k(\zz)=0$ for all $k\geq 0$,
the limit $\lim_{k}V_{\hpi^k}(\zz)$ does a.s. not exist.

\vip

{\it Step 1.} We introduce the zones, see Figure 1,
\begin{align*}
Z_1=&\Big\{(u,v) \in \rr^2 : v>u>1 \quad \hbox{and}\quad  
v \in \Big(\frac 1{2(1-\gamma)},\frac{1+2\gamma}{4(1-\gamma)}\Big)\Big\},\\
Z_2=&\Big\{(u,v) \in \rr^2 :v<u<\frac\gamma{1-\gamma}\quad \hbox{and}\quad  
v \in \Big(\frac 1{2(1-\gamma)},\frac{1+2\gamma}{4(1-\gamma)}\Big)\Big\},\\
Z_3=&\Big\{(u,v) \in \rr^2 : v<u<\frac\gamma{1-\gamma}\quad \hbox{and}\quad  
v \in \Big(\frac{3-2\gamma}{4(1-\gamma)},\frac 1{2(1-\gamma)}\Big)\Big\},\\
Z_4=&\Big\{(u,v) \in \rr^2 : v>u>1  \quad \hbox{and}\quad  
v \in \Big(\frac{3-2\gamma}{4(1-\gamma)},\frac 1{2(1-\gamma)}\Big)\Big\}.
\end{align*}
\begin{figure}[t]
\noindent\fbox{\begin{minipage}{0.95\textwidth}
\centerline{\includegraphics[width=8cm]{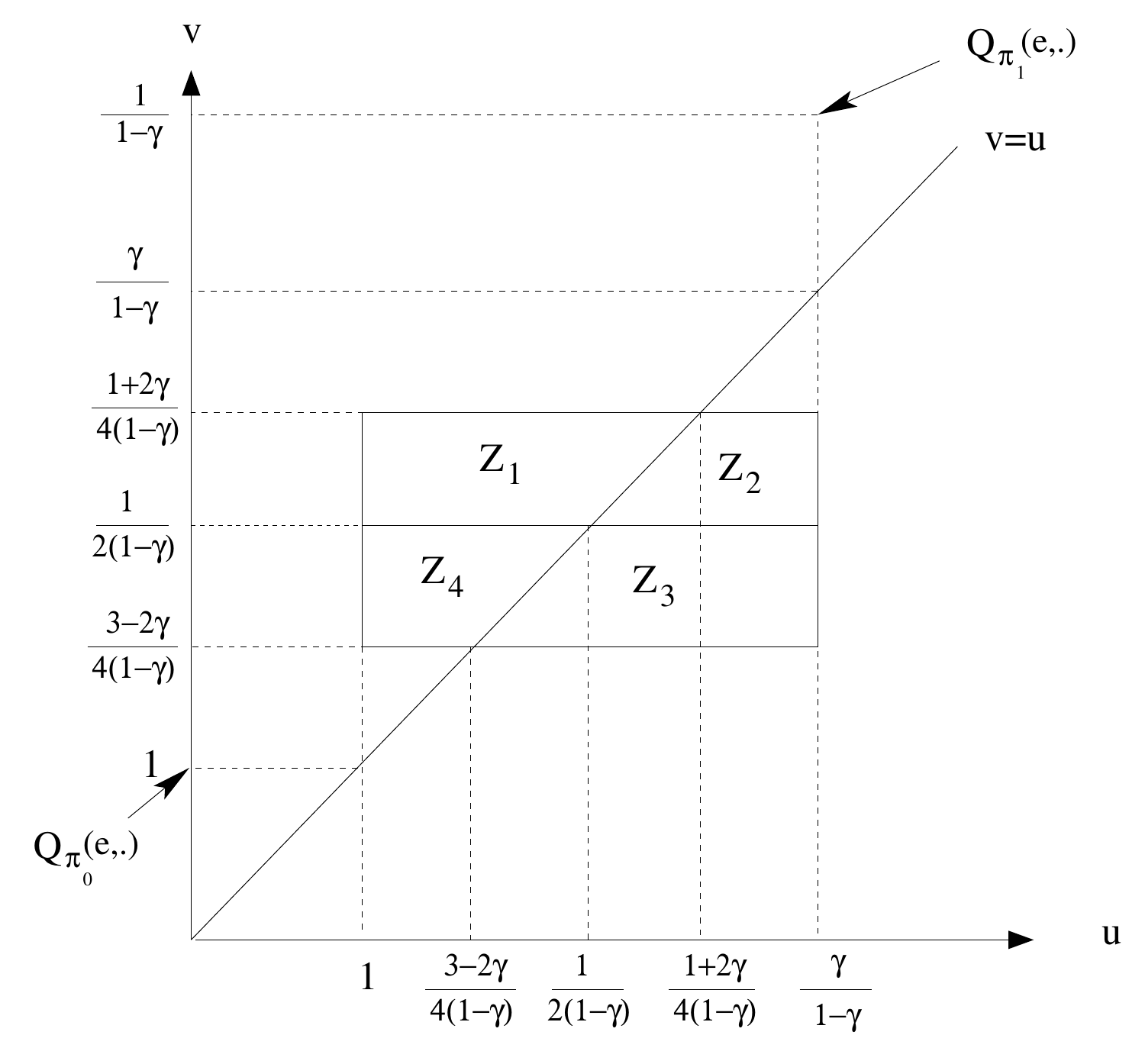}}
\caption{\label{fig}}
\end{minipage}}
\end{figure}
We choose $(u_0,v_0)$ in $Z_1$ such that $u_0 \in \Q$ and $v_0 \in \rr\setminus\Q$ and
set $\hQ^0=(u_0,v_0)$.
We also fix 
\begin{equation}\label{qq}
q \in \Q \cap \Big(0,\frac\gamma 2 - \frac14\Big),
\quad \hbox{whence}\quad q< \frac14\land \Big(\gamma -\frac12\Big)\land \Big(\frac\gamma2 -\frac14\Big)
\land \frac{2\gamma-1}{3-2\gamma}.
\end{equation}
For all $k\geq 0$ and for $a\in \{0,1\}$, we define
(once the $k-1$ first episodes have been built)
$$
\alpha_k(\zz,a)=\frac{q\rho_k(a)}{1+L_k(a)}, \quad\hbox{where}\quad L_k(a)=\sum_{i=0}^k \rho_i(a)
\quad\hbox{and}\quad \rho_i(a)=\left\{\begin{array}{lll}
\indiq_{\{a=0\}} & \hbox{if}& \hQ^i \in Z_1\cup Z_3, \\
\indiq_{\{a=1\}} & \hbox{if}& \hQ^i \in Z_2\cup Z_4,
\end{array}\right.
$$ 
and we set
$$
\blue \nu_k(\zz,a)=\rho_k(a).
$$

{\it Step 2.} 
We now check that for all $k\geq 0$, all $i \in \{1,2,3,4\}$, 
if $\hQ^k=(u_k,v_k)\in Z_i$ with $u_k \in \Q$ and $v_k\in \rr\setminus \Q$, then, with the convention that
$Z_5=Z_1$, 
\vip
$\bullet$ $\hQ^{k+1}=(u_{k+1},v_{k+1})\in Z_i \cup Z_{i+1}$ with also
$u_{k+1} \in \Q$ and $v_{k+1}\in \rr\setminus \Q$;
\vip
$\bullet$ there is $\ell>k$ such that $\hQ^{\ell} \in Z_{i+1}$.
\vip
This in particular implies
that $\hQ^k$ belongs to $Z_1\cup Z_2\cup Z_3\cup Z_4$ for all $k\geq 0$, so that 
the definitions of $\alpha_k$ and $\nu_k$ are sufficient to produce
the whole algorithm.

\vip

\blue Here we use that $\gamma \in \Q$, that $u_0 \in \Q$ and $v_0 \notin \Q$ to guarantee that for all $k\geq 0$,
$(u_k,v_k)$ never lies on the diagonal (as we will see, $u_k \in \Q$ and $v_k \notin \Q$ for all $k\geq 0$). 
Else, the policy of Algorithm 1 would be randomized, leading to some complications.
\black 
\vip
(a) If first $\hQ^k \in Z_1$, then $\hpi^k=\pi_1$ 
(recall \eqref{gg1} and that $\e_k\equiv 0$).
Moreover, we have $\nu_k(\zz,a)=\indiq_{\{a=0\}}$. Hence $\tau_{\zz,0}^{k+1}=0$
and  $\tau_{\zz,1}^{k+1}=1$, and $G_{\zz,0}^{k+1}=\gamma/(1-\gamma)$ and $G_{\zz,1}^{k+1}=1/(1-\gamma)$.
By \eqref{gg3} and by definition of $\alpha_k$, we find
\begin{equation}\label{aaa1}
u_{k+1}=(1-\alpha_k(\zz,0))u_k+\alpha_k(\zz,0)\frac\gamma{1-\gamma}
=u_k + \frac q{1+L_k(0)}\Big(\frac\gamma{1-\gamma}-u_k\Big)\quad\hbox{and}\quad v_{k+1}=v_k.
\end{equation}
Hence $u_{k+1}\in\Q$ and $v_{k+1}\in\rr\setminus\Q$, recall that $\gamma \in \Q$ by assumption. 
Moreover, $\hQ^{k+1} \in Z_1\cup Z_2$, because $v_{k+1}\neq u_{k+1}$ and, using that 
$u_k \in (1,(1+2\gamma)/[4(1-\gamma)])$ and that $q<1/4$, see \eqref{qq},
\begin{gather*}
v_{k+1}=v_k \in \Big(\frac 1{2(1-\gamma)},\frac{1+2\gamma}{4(1-\gamma)}\Big) ,\\
1< u_{k}<u_{k+1}< u_k+q \Big(\frac \gamma{1-\gamma}-u_k\Big) < \frac{1+2\gamma}{4(1-\gamma)}
+q \frac {2\gamma-1}{1-\gamma} < \frac \gamma{1-\gamma}.
\end{gather*}
Assume next that $\hQ^\ell \in Z_1$ for all $\ell \geq k$, then for all $n \geq k$,
we would have $L_n(0)=L_{k}(0)+\ell-k$ and, by \eqref{aaa1} and since 
$u_n <(1+2\gamma)/[4(1-\gamma)]$ for all $n\geq k$, for all $\ell\geq k$, 
$$
u_\ell=u_{k}+  \sum_{n=k}^{\ell-1} \Big(\frac\gamma{1-\gamma}-u_n\Big)\frac{q}{1+L_n(0)}
\geq u_{k}+ \frac{(2\gamma-1)q}{4(1-\gamma)}\sum_{n=k}^{\ell-1} \frac1{1+L_{k}(0)+n-k}.
$$
This would imply that $u_\ell \to \infty$ as $\ell\to \infty$ and contradict 
the fact that  $\hQ^\ell \in Z_1$ for all $\ell\geq k$.
\vip

(b) If next $\hQ^k \in Z_2$, then $\hpi^k=\pi_0$ and $\nu_k(\zz,a)=\indiq_{\{a=1\}}$. Hence
$\tau_{\zz,0}^{k+1}=1$ and  $\tau_{\zz,1}^{k+1}=0$, while $G_{\zz,0}^{k+1}=0$ and $G_{\zz,1}^{k+1}=1$.
By \eqref{gg3} and by definition of $\alpha_k$, we find
\begin{equation}\label{aaa2}
u_{k+1}=u_k \qquad \hbox{and}\qquad  v_{k+1}=(1-\alpha_k(\zz,1))v_k+\alpha_k(\zz,1)\times 1
=v_{k}- \frac q{1+L_k(1)} (v_k-1).
\end{equation}
Thus $u_{k+1}\in\Q$ and $v_{k+1}\in\rr\setminus\Q$.
Moreover, $\hQ^{k+1} \in Z_2\cup Z_3$, because $u_{k+1}\neq v_{k+1}$ and, 
since $v_k>1/[2(1-\gamma)]$ and since $q<\gamma-1/2$, see \eqref{qq},
\begin{gather*}
u_{k+1}=u_k \in \Big(\frac 1{2(1-\gamma)},\frac{\gamma}{1-\gamma}\Big) ,\\
u_{k+1}=u_k>v_{k}> v_{k+1}>  v_k-q(v_k-1) > v_k(1-q)> \frac{1-q}{2(1-\gamma)}>\frac{3-2\gamma}{4(1-\gamma)}.
\end{gather*}
If $\hQ^\ell \in Z_2$ for all $\ell\geq k$, then for all $n\geq k$, we would have
$L_n(1)=L_{k}(1)+n-k$ and, by \eqref{aaa2} and since $v_n >1/[2(1-\gamma)]$ for all $n\geq k$,
for all $\ell \geq k$,
$$
v_\ell=v_{k}- \sum_{n=k}^{\ell-1} \frac{q (v_n-1)}{1+L_n(1)}
\leq v_{k}- \frac{(2\gamma-1)q}{2(1-\gamma)}\sum_{n=k}^{\ell-1} \frac{1}{1+L_{k}(1)+n-k}.
$$
This would imply that $v_\ell \to -\infty$ as $\ell\to \infty$ and contradict 
the fact that  $\hQ^\ell \in Z_2$ for all $\ell\geq k$.
\vip

(c) If now $\hQ^k \in Z_3$, then $\hpi^k=\pi_0$ and $\nu_k(\zz,a)=\indiq_{\{a=0\}}$. Hence
$\tau_{\zz,0}^{k+1}=0$ and  $\tau_{\zz,1}^{k+1}=\infty$, and $G_{\zz,0}^{k+1}=0$.
By \eqref{gg3} and by definition of $\alpha_k$, we find
\begin{equation}\label{aaa3}
u_{k+1}=(1-\alpha_k(\zz,0))u_k+\alpha_k(\zz,0)\times 0 =u_{k}- \frac q{1+L_k(0)} u_k
\qquad \hbox{and}\qquad  v_{k+1}=v_k.
\end{equation}
Thus $u_{k+1}\in\Q$ and $v_{k+1}\in\rr\setminus\Q$.
Moreover, $\hQ^{k+1} \in Z_3\cup Z_4$, because $v_{k+1}\neq u_{k+1}$ and, 
since $u_k>(3-2\gamma)/[4(1-\gamma)]$ and $q<(2\gamma-1)/(3-2\gamma)$, see \eqref{qq},
\begin{gather*}
v_{k+1}=v_k \in \Big(\frac {3-2\gamma}{4(1-\gamma)},\frac1{2(1-\gamma)}\Big) ,\\
\frac\gamma{1-\gamma}>u_{k}>u_{k+1}> u_k(1-q)>\frac{(1-q)(3-2\gamma)}{4(1-\gamma)}>1.
\end{gather*}
If $\hQ^\ell \in Z_3$ for all $\ell\geq k$, then for all $n\geq k$, we would have
$L_n(0)=L_{k}(0)+n-k$ and, by \eqref{aaa3} and  since
$u_n >1$ for all $n\geq k$, for all $\ell \geq k$,
$$
u_\ell=u_{k}-  \sum_{n=k}^{\ell-1} \frac{q u_n}{1+L_n(0)}
\leq u_{k}- q\sum_{n=k}^{\ell-1} \frac1{1+L_{k}(0)+n-k}.
$$
This would imply that $u_\ell \to -\infty$ as $\ell\to \infty$ and contradict 
the fact that  $\hQ^\ell\in Z_3$ for all $\ell\geq k$.

\vip
(d) If finally $\hQ^k \in Z_4$, then $\hpi^k=\pi_1$ and $\nu_k(\zz,a)=\indiq_{\{a=1\}}$. Hence
$\tau_{\zz,0}^{k+1}=\infty$ and  $\tau_{\zz,1}^{k+1}=0$, and $G_{\zz,1}^{k+1}=\frac 1{1-\gamma}$.
By \eqref{gg3} and by definition of $\alpha_k$, we find
\begin{equation}\label{aaa4}
u_{k+1}=u_k \qquad \hbox{and}\qquad  v_{k+1}=(1-\alpha_k(\zz,1))v_k + \frac{\alpha_k(\zz,1)}{1-\gamma}
=v_k + \frac{q}{1+L_k(1)}\Big(\frac 1{1-\gamma}-v_k\Big).
\end{equation}
Thus $u_{k+1}\in\Q$ and $v_{k+1}\in\rr\setminus\Q$.
Moreover, $\hQ^{k+1} \in Z_4\cup Z_1$, because  $v_{k+1}\neq u_{k+1}$ and, since $v_k <1/[2(1-\gamma)]$
and since $q<\gamma/2-1/4$, see \eqref{qq},
\begin{gather*}
u_{k+1}=u_k \in \Big(1,\frac1{2(1-\gamma)}\Big) ,\\
1<u_{k+1}=u_k<v_k<v_{k+1}<v_k+ q\Big(\frac 1{1-\gamma}-v_k\Big)
<\frac1{2(1-\gamma)}+\frac q{1-\gamma} <\frac{1+2\gamma}{4(1-\gamma)}.
\end{gather*}
If  $\hQ^\ell \in Z_4$ for all $\ell\geq k$, then for all $n\geq k$, we would have
$L_n(1)=L_{k}(1)+n-k$ and, by \eqref{aaa4} and  since
$v_n <1/[2(1-\gamma)]$ for all $n\geq k$, for all $\ell \geq k$,
$$
v_\ell=v_{k}+ \sum_{n=k}^{\ell-1} \frac{q}{1+L_n(1)}\Big(\frac 1{1-\gamma}-v_n\Big)
\geq v_{k} + \frac q{2(1-\gamma)}\sum_{n=k}^{\ell-1} \frac1{1+L_{k}(1)+n-k}.
$$
This would imply that $v_\ell \to \infty$ as $\ell\to \infty$ and contradict the fact that  
$\hQ^\ell \in Z_4$ for all $\ell\geq k$.
\vip

{\it Step 3.} We conclude that $V_{\hpi^k}(\zz)$ does not converge as $k\to \infty$,
because $V_{\hpi^k}(\zz)=V_{\pi_1}(\zz)=1/(1-\gamma)$ for those $k$'s such that $\hQ^k \in Z_1\cup Z_4$
(there are infinitely many such $k$'s by Step 2),
while  $V_{\hpi^k}(\zz)=V_{\pi_0}(\zz)=0$ for those $k$'s such that $\hQ^k \in Z_2\cup Z_3$
(there are infinitely many such $k$'s by Step 2).
It only remains to show \eqref{cqdev}. But, recalling the definition of $\alpha_k$
and that $\tau^{k+1}_{\zz,0}<\infty$ when $\hQ^k \in Z_1\cup Z_3$, see Steps 2-(a)-(c),
we find
$$
\sum_{k\geq 0} \alpha_k(\zz,0)\indiq_{\{\tau^{k+1}_{\zz,0}<\infty\}} = \sum_{k\geq 0} \frac{q \indiq_{\{\hQ^k \in Z_1\cup Z_3\}}}
{1+\sum_{i=0}^k\indiq_{\{\hQ^i \in Z_1\cup Z_3\}} } = \sum_{\ell \geq 1} \frac{q }{1+\ell}=\infty.
$$
We used that $\sum_{i=0}^k\indiq_{\{\hQ^i \in Z_1\cup Z_3\}}=\infty$ by Step 2. Next,
$$
\sum_{k\geq 0} (\alpha_k(\zz,0))^2\indiq_{\{\tau^{k+1}_{\zz,0}<\infty\}} = \sum_{k\geq 0} \frac{q^2 \indiq_{\{\hQ^k \in Z_1\cup Z_3\}}}
{(1+\sum_{i=0}^k\indiq_{\{\hQ^i \in Z_1\cup Z_3\}})^2} = \sum_{\ell \geq 1} \frac{q^2 }{(1+\ell)^2}<\infty.
$$
Similarly, using that $\tau^{k+1}_{\zz,1}<\infty$ when $\hQ^k \in Z_2\cup Z_4$, see Steps 2-(b)-(d),
$$
\sum_{k\geq 0} \alpha_k(\zz,1)\indiq_{\{\tau^{k+1}_{\zz,1}<\infty\}} = \sum_{k\geq 0} \frac{q \indiq_{\{\hQ^k \in Z_2\cup Z_4\}}}
{1+\sum_{i=0}^k\indiq_{\{\hQ^i \in Z_2\cup Z_4\}} } = \sum_{\ell \geq 1} \frac{q }{1+\ell}=\infty,
$$
because $\sum_{i=0}^k\indiq_{\{\hQ^i \in Z_2\cup Z_4\}}=\infty$ by Step 2, and
$$
\sum_{k\geq 0} (\alpha_k(\zz,1))^2\indiq_{\{\tau^{k+1}_{\zz,1}<\infty\}} = \sum_{k\geq 0} \frac{q^2 \indiq_{\{\hQ^k \in Z_2\cup Z_4\}}}
{(1+\sum_{i=0}^k\indiq_{\{\hQ^i \in Z_2\cup Z_4\}})^2 } = \sum_{\ell \geq 1} \frac{q^2 }{(1+\ell)^2}<\infty.
$$
The proof is complete.
\end{proof}

\begin{rk}\label{rkHH}
\blue When $\gamma \in (1/2,1)$ does not belong to $\Q$, one can adapt the above proof using the 
following trick: consider some
dense subset $\HH$ of $\rr$ such that $\rr\setminus\HH$ is also dense and such that for any 
$p \in \Q^*$, any $p',p'' \in \Q$, for any $u,v \in \rr$,
\vip
\noindent $\bullet$ $u \in \HH$ implies that $p u + p' \gamma/(1-\gamma) \in \HH$,
\vip
\noindent $\bullet$ $v \notin \HH$ implies that 
$pv+p'+p''/(1-\gamma) \notin \HH$.
\vip
If  $u_0 \in \HH$ and $v_0 \notin \HH$, this will imply that for all $k\geq 1$, $u_k \in \HH$ 
and $v_k \notin \HH$, so that $u_k\neq v_k$.
\vip
Such a set can be defined as $\HH=\{q+q'/(1-\gamma)+q'' \gamma/(1-\gamma) : q,q',q'' \in \Q\}$, which 
is dense since it contains $\Q$ and of which the complementary set is dense since $\HH$ is countable. 
The other properties are easily checked.
\black
\end{rk}

\appendix
\section{Quick proofs of know results}\label{qpkr}

Here we recall, for the sake of completeness, the 

\begin{proof}[Proof of Theorem \ref{known}.] We adopt the notation
introduced in Subsections \ref{mod} and \ref{opol}.
\vip

{\it Step 1.} For $f:\cX\to\rr$ and $x\in \cX$, we set $\cT(f)(x)=\max_{a\in\cA_x}[r(x,a)+\gamma Pf(x,a)]$,
where $Pf(x,a)=\sum_{y\in \cX} P(x,a,y)f(y)$.
For $f,g:\cX\to\rr$, it holds that
$$
||\cT(f)-\cT(g)||_\infty= \gamma \max_{x \in \cX, a\in \cA_x}|Pf(x,a)-Pg(x,a)|\leq \gamma ||f-g||_\infty.
$$
Since $\gamma \in [0,1)$, we conclude that 
there exists a unique function $\tV : \cX \to \rr$ such that $\cT(\tV)=\tV$.
\vip

{\it Step 2.} For $N\geq 1$, set $G_N=\sum_{t=0}^{N-1}\gamma^t R_{t+1}$.
For any policy $\Pi$, any $x \in \cX$, set $V_\Pi^N(x)=\E_{x,\Pi}[G_N]$. 
Let ${\bf 0}:\cX\to\rr$ be the null function.
For all $x\in \cX$, it holds that 
$$
V^N_\Pi(x) \leq \cT^{\circ N}({\bf 0})(x).
$$ 
\blue Here $\cT^{\circ N}=\cT\circ\cdots\circ\cT$ with $N-1$ circles. \black
\vip

It suffices to show that for all $k\in \ig 1, N \id$,
$$
\E_{x,\Pi}\Big[\sum_{t=N-k}^{N-1} \gamma^t R_{t+1}\Big|H_{N-k},A_{N-k}\Big]\leq \gamma^{N-k} \cT^{\circ k}({\bf 0})(X_{N-k}),
$$
recall that $H_t$ was defined in Subsection \ref{mod}. With $k=N$, this gives the result. First, when $k=1$,
$$
\E_{x,\Pi}\Big[\gamma^{N-1}R_{N}\Big|H_{N-1},A_{N-1}\Big]=\gamma^{N-1}r(X_{N-1},A_{N-1}) 
\leq \gamma^{N-1} \cT({\bf 0})(X_{N-1}).
$$
Next, assuming that the inequality holds true for some $k \in \ig 1, N-1 \id$,
\begin{align*}
&\E_{x,\Pi}\Big[\sum_{t=N-k-1}^{N-1} \gamma^t R_{t+1}\Big|H_{N-k-1},A_{N-k-1}\Big]\\
\leq & \E_{x,\Pi}[\gamma^{N-k-1}R_{N-k}|H_{N-k-1},A_{N-k-1}]
+\E_{x,\Pi}[\gamma^{N-k} \cT^{\circ k}({\bf 0})(X_{N-k}) |H_{N-k-1},A_{N-k-1}]\\
=  & \gamma^{N-k-1}\Big(r(X_{N-k-1},A_{N-k-1}) + \gamma P \cT^{\circ k}({\bf 0})(X_{N-k-1},A_{N-k-1}) \Big)\\
\leq & \gamma^{N-k-1}\cT^{\circ (k+1)}({\bf 0})(X_{N-k-1}).
\end{align*}

{\it Step 3.} We show here that for all policy $\Pi$, all $x\in \cX$, we have $V_\Pi(x) \leq \tV(x)$. 
This of course implies that $V^*(x) \leq \tV(x)$ for all $x\in \cX$.

\vip
We have $V_\Pi(x)=\lim_{N\to \infty} V^N_\Pi(x)$, because, recalling Setting \ref{set},
$$
|V_\Pi(x)-V^N_\Pi(x)|\leq 
\sum_{t\geq N} \gamma^t |\E_{x,\Pi}[R_N]| \leq ||g||_\infty \sum_{t\geq N} \gamma^t \to 0.
$$
Hence by Step 2, $V_\Pi(x)\leq \lim_{N\to \infty} \cT^{\circ N}({\bf 0})(x)$.
This last quantity equals $\tV(x)$ by Step 1.
\vip

{\it Step 4.} For any SM policy $\pi$, for $f:\cX\to \rr$ and for $x \in \cX$, we set 
$T_\pi(f)(x)=r_\pi(x)+\gamma P_\pi f(x)$, recall \eqref{rpiPpi}. 
We check in this step that $V_\pi$ is the unique fixed point of $T_\pi$.

\vip
First, $T_\pi$ is a contraction, since $||T_\pi f -T_\pi g||_\infty= \gamma ||P_\pi f -P_\pi g||_\infty
\leq \gamma ||f -g||_\infty$. 
Next,
$$
V_\pi(x)=\E_{x,\pi}\Big[R_1+ \sum_{t\geq 1}\gamma^t R_{t+1}\Big]
=r_\pi(x)+\gamma \E_{x,\pi}\Big[ \E_{x,\pi}\Big[ \sum_{t\geq 1}\gamma^{t-1} R_{t+1}\Big| H_1\Big]\Big]
=r_\pi(x)+\gamma \E_{x,\pi}[V_\pi(X_1)].
$$
Since $\E_{x,\pi}[V_\pi(X_1)]=P_\pi V_\pi(x)$, we conclude that $V_\pi(x)=r_\pi(x)+\gamma P_\pi V_\pi(x)=T_\pi(V_\pi)(x)$.
\vip

{\it Step 5.} Let $\pi^*$ be a SM policy satisfying 
\begin{equation}\label{trruc2}
\pi^*\Big(x, \argmax \tilde Q (x,\cdot)\Big)=1 \quad \hbox{for all $x \in \cX$, where} \quad 
\tilde Q(x,a)=r(x,a)+\gamma P \tV (x,a).
\end{equation}
Here we prove that then, 
$\tV = T_{\pi^*}(\tV)$. By Step 4, we will conclude that $\tV=V_{\pi^*}$, whence $V_{\pi^*}\geq V^*$ by Step 3.
By definition of $V^*$, this implies that $V_{\pi^*}=V^*=\tV$.
\vip
By definition, see Step 1, we have, for all $x \in \cX$,
$$
\tV(x)=\max_{a \in \cA_x} [r(x,a)+ \gamma P \tV (x,a)]= \sum_{a \in \cA_x}[r(x,a)+ \gamma P \tV (x,a)]\pi^*(x,a)
$$
by \eqref{trruc2}. Hence 
$$
\tV(x)=r_{\pi^*}(x)+\gamma\sum_{a \in \cA_x} \sum_{y \in \cX}P(x,a,y)\tV(y)\pi^*(x,a)=
r_{\pi^*}(x)+\gamma\sum_{y \in \cX} \Big(\sum_{a \in \cA_x} P(x,a,y)\pi^*(x,a)\Big) \tV(y),
$$
i.e. $\tV(x)=r_{\pi^*}(x)+ \gamma P_{\pi^*} \tV (x)=T_{\pi^*}(\tV)(x)$.
\vip

{\it Conclusion.} By Step 5, we know that $\tV=V^*$. Hence $\tilde Q=Q^*$ (see \eqref{qstar} and \eqref{trruc2}), 
so that $\pi^*$ satisfies \eqref{trruc}
if and only if it satisfies \eqref{trruc2}. Such a SM policy $\pi^*$ satisfies $V_{\pi^*}(x)=V^*(x)$
for all $x \in \cX$ by Step 5, which shows (i). 

\vip

Let now $\pi$ by any SM policy. Recalling \eqref{qpi}, we have $Q_\pi=r+\gamma PV_\pi\leq
r+\gamma PV^*=Q^*$, see \eqref{qstar}.
We have seen that $V_\pi=T_\pi(V_\pi)=r_\pi+\gamma P_\pi V_\pi$ in Step 4. Moreover, we have
$$
\sum_{a\in \cA_x}Q_\pi(x,a)\pi(x,a)=\sum_{a\in \cA_x}r(x,a)\pi(x,a)+\gamma \sum_{a\in \cA_x}PV_\pi(x,a)\pi(x,a)=
r_\pi(x)+\gamma P_\pi V_\pi(x)=V_\pi(x).
$$
Finally, 
$$
Q_\pi(x,a)=r(x,a)+\gamma\sum_{y \in \cX} P(x,a,y)V_\pi(y)=r
(x,a)+\gamma\sum_{y \in \cX} P(x,a,y)\sum_{b\in \cA_y}Q_\pi(y,b)\pi(y,b),
$$
and we have checked (ii).

\vip
If $\pi^*$ satisfies \eqref{trruc} (or equivalently \eqref{trruc2}), then 
$Q_{\pi^*}=r+\gamma PV_{\pi^*}=r+\gamma PV^*=Q^*$ because $V_{\pi^*}=V^*$. Moreover, recalling Steps 1 and 5,
\begin{align*}
V^*(x)=\tV(x)=\max_{a\in \cA_x}[r(x,a)+\gamma P \tV(x,a)]=\max_{a\in \cA_x}[r(x,a)+\gamma P V^*(x,a)]
=\max_{a\in \cA_x} Q^*(x,a),
\end{align*}
and, by point (ii) applied to $\pi^*$,
$$
Q^*(x,a)=r(x,a)+\gamma\sum_{(y,b)\in\cZ}P(x,a,y)\pi^*(y,b)Q^*(y,b)=r(x,a)+\gamma\sum_{y\in\cX}P(x,a,y)\max_{b\in \cA_y}
Q^*(y,b)
$$
by \eqref{trruc}. This proves (iii). 
\end{proof}

Finally, we recall the

\begin{proof}[Proof of Lemma \ref{rm}]
A simple computation shows that for all $k\geq 0$,
\begin{equation}\label{ttaacc}
\E[Z_{k+1}^2 |\cG_k] \leq (1-\theta_k)^2Z_k^2 + C \theta_k^2 = (1+\theta_k^2)Z_k^2+C \theta_k^2-2\theta_kZ_k^2.
\end{equation}
We set $\gamma_0=1$ and $M_0=Z_0$, $N_0=Z_0$. For $k\geq 1$, we set 
$\gamma_k=[\prod_{\ell=0}^{k-1} (1+\theta_\ell^2)]^{-1}$
and introduce 
$$
M_k=\gamma_k Z_k^2 - C \sum_{\ell=0}^{k-1} \gamma_{\ell+1}\theta_\ell^2 \qquad
\hbox{and} \qquad N_k=M_k+2 \sum_{\ell=0}^{k-1} \gamma_{\ell+1}\theta_\ell Z_\ell^2.
$$
One easily checks, using \eqref{ttaacc} and that $\gamma_{k+1}$ is $\cG_{k}$-measurable, 
that $(M_k)_{k\geq 0}$ and $(N_k)_{k\geq 0}$ are two supermartingales.
Let now $L:=\sum_{k=0}^{\infty} \gamma_{k+1}\theta_k^2\leq \sum_{k\geq 0} \theta_k^2$, which is a.s. finite by assumption.
Since $N_k\geq M_k \geq -CL$ for all $k\geq 0$, both $M_\infty=\lim_k M_k \in \rr$
and $N_\infty=\lim_k N_k \in \rr$ a.s. exist. Using again  that 
$\sum_{k\geq 0} \theta_k^2<\infty$, we deduce that $\gamma_\infty=\lim_{k} \gamma_k>0$ a.s.
We first conclude that
$$
\ell:=\lim_k Z_k^2 = \frac{M_\infty + C L }{\gamma_\infty} \in \rr_+
$$
a.s. exists, and next that 
$$
2\sum_{k\geq 0} \gamma_{k+1}\theta_k Z_k^2=N_\infty-M_\infty \in \rr_+ \quad \hbox{a.s.}
$$
This tells us that necessarily $\ell=0$ a.s. because $\gamma_\infty=\lim_{k} \gamma_k>0$ a.s. and 
$\sum_{k\geq 0} \theta_k=\infty$ a.s. by assumption.
\end{proof}


\begin{thebibliography}{99}
\blue
\bibitem{baird}{{\sc L. Baird}, {\it Residual Algorithms: Reinforcement Learning with Function Approximation.}
In {\it Proceedings of the 12-th International Conference on Machine Learning,} 1995, 30--37. Morgan Kaufmann.}
\black
\bibitem{bt}{{\sc D.P. Bertsekas, J.N. Tsitsiklis.} Neuro-Dynamic Programming. Athena Scientific, Belmont, 
Massachusetts, 1996.}

\blue

\bibitem{d}{{\sc P. Dayan.} {\it The convergence of TD$(\lambda)$ for general $\lambda$.} 
Machine learning 8 (1992), 341--362.}

\black

\bibitem{l}{{\sc J. Liu.} {\it On the convergence of reinforcement learning with Monte Carlo Exploring Starts.}
Automatica 129 (2021), 109693.}

\bibitem{p}{{\sc M.L. Puterman.} Markov decision processes: discrete stochastic dynamic programming. 
Wiley Series in Probability and Mathematical Statistics: Applied Probability and Statistics. 
A Wiley-Interscience Publication. John Wiley \& Sons, Inc., New York, 1994.}

\bibitem{rm}{{\sc H. Robbins, S. Monro.} {\it A stochastic approximation method.}
Ann. Math. Statistics 22 (1951), 400--407.}

\bibitem{sjls}{{\sc S. Singh, T. Jaakkola, M.L. Littman, C. Szepesv\'ari.}
{\it Convergence Results for Single-Step On-Policy Reinforcement-Learning Algorithms}.
Machine Learning 39 (2000), 287--308.}

\bibitem{ss}{{\sc S.P. Singh, R.S. Sutton.} {\it Reinforcement learning with replacing eligibility traces.}
Machine Learning, 22(1-3):123--158.}

\blue

\bibitem{silv}{{\sc D. Silver.} Lecture 6: Value Function Approximation. 
https://davidstarsilver.wordpress.com/teaching/}

\bibitem{s}{{\sc R.S. Sutton.}  {\it Learning to predict by the methods of temporal differences}, Machin Learning
3 (1988), 9--44.}
\black

\bibitem{sb}{{\sc R.S. Sutton, A.G. Barto.} Reinforcement learning: an introduction. Second edition. 
Adaptive Computation and Machine Learning. MIT Press, Cambridge, MA, 2018.}

\bibitem{t}{{\sc J.N. Tsitsiklis.} {\it On the convergence of optimistic policy iteration.} 
J. Mach. Learn. Res. 3 (2003), no. 1, 59--72.}

\bibitem{wysr}{{\sc C. Wang, S. Yuan, K. Shao, K. Ross.} 
{\it On the Convergence of the Monte Carlo Exploring Starts Algorithm for Reinforcement Learning.}
ICLR 2022 Conference.}

\blue

\bibitem{wd}{{\sc C. Watkins, P. Dayan.} {\it $Q$-learning.} Machine Learning 8 (1992), 279--292.}

\black

\end{thebibliography}
\end{document}